\documentclass[reqno]{amsart}
\usepackage[a4paper,hmarginratio=1:1]{geometry}

\usepackage{amsmath}
\usepackage{amssymb}
\usepackage{mathtools}
\usepackage{amsthm}
\usepackage{graphicx}
\usepackage{titlesec}
\usepackage{lipsum}
\usepackage{enumitem}
\usepackage{mathabx}
\usepackage{tikz-cd}
\usepackage{dsfont}
\usepackage{lmodern}
\usepackage{bbm}
\usepackage{adjustbox}
\usepackage{hyperref}
\usepackage{mathrsfs}
\definecolor{dark-red}{rgb}{0.5,0.15,0.15}
\definecolor{dark-blue}{rgb}{0.15,0.15,0.6}
\definecolor{dark-green}{rgb}{0.15,0.6,0.15}
\hypersetup{
    colorlinks, linkcolor=dark-red,
    citecolor=dark-blue, urlcolor=dark-green
}
\usepackage[capitalize]{cleveref}
\newtheorem{theorem}{Theorem}[section]
\newtheorem{lemma}[theorem]{Lemma}
\newtheorem{proposition}[theorem]{Proposition}
\newtheorem{corollary}[theorem]{Corollary}
\theoremstyle{definition}
\newtheorem{definition}[theorem]{Definition}
\newtheorem{example}[theorem]{Example}

\newtheorem{remark}[theorem]{Remark}

\newtheorem{notation}[theorem]{Notation}

\numberwithin{equation}{section}

\newcommand{\gp}{\underline{\textup{GP}}}
\newcommand{\gi}{\overline{\textup{GI}}}

\newcommand{\lhf}{\textup{LH}\mathfrak{F}}
\newcommand{\lhx}{\textup{LH}\mathfrak{X}}
\newcommand{\lhg}{\textup{LH}\mathfrak{G}}

\newcommand{\tacstab}{\widetilde{\textup{Mod}}}

\newcommand{\tachom}{\widetilde{\textup{Hom}}}

\titleformat{\section}
  {\normalfont\normalsize\bfseries}{\thesection.}{1em}{}
  \titleformat{\subsection}
  {\normalfont\normalsize\bfseries}{\thesubsection.}{1em}{}
\titleformat{\subsubsection}
  {\normalfont\normalsize\bfseries}{\thesubsubsection.}{1em}{}
\title[Stable module category]{The stable module category and model structures for hierarchically defined groups}
\author{Gregory Kendall}
\begin{document}
\begin{abstract}
    In this work we construct a compactly generated tensor-triangulated stable category for a large class of infinite groups, including those in Kropholler's hierarchy $\mathrm{LH}\mathfrak{F}$. This can be constructed as the homotopy category of a certain model category structure, which we show is Quillen equivalent to several other model categories, including those constructed by Bravo, Gillespie, and Hovey in their work on stable module categories for general rings. We also investigate the compact objects in this category. In particular, we give a characterisation of those groups of finite global Gorenstein AC-projective dimension such that the trivial representation $\mathbb{Z}$ is compact. 
\end{abstract}
\date{\today}
\maketitle
\section{Introduction}
For finite groups, the stable module category is usually constructed by quotienting out by morphisms which factor through a projective module and it has the structure of a rigidly-compactly generated tensor-triangulated category. For infinite groups, this does not in general even form a triangulated category. In order to construct an appropriate stable category for infinite groups one must decide which properties to generalise. 
\par 
Our main motivation is to construct a compactly generated, tensor-triangulated category for a class of infinite groups including those in Kropholler's hierarchy $\lhf$ \cite{krop}. We work with the class of $\lhg$ and $\textup{LH}\mathfrak{G}_{AC}$ groups, where $\mathfrak{G}$ and $\mathfrak{G}_{AC}$ are the class of groups of finite Gorenstein cohomological dimension and finite global Gorenstein AC-projective dimension respectively. The following is a combination of \Cref{ourbigmodel}, \Cref{ttcatth}, and \Cref{somecge}. 
\begin{theorem}
    Let $k$ be a commutative noetherian ring of finite global dimension and $G$ an $\lhg$ group. Then there is an abelian model structure on $\textup{Mod}(kG)$ such that the homotopy category, which is equivalent to the stable category of Gorenstein projectives $\gp(kG)$, is a tensor-triangulated category. If $G$ is an $\textup{LH}\mathfrak{G}_{AC}$ group, then the homotopy category is compactly generated.
\end{theorem}
We note that the difficult part of showing the tensor-triangulated structure is in showing that there exists a tensor unit. We do this by showing the existence of Gorenstein projective special precovers for $\lhg$ groups, see \Cref{induc}. For a general ring it is an open question if such special precovers have to exist, although this has been shown in the affirmative by Cortés-Izurdiaga and \v{S}aroch \cite{cortésizurdiaga2023cotorsion} given an additional large cardinal assumption.
\par 
We make a few remarks on the classes of groups with which we work. Every $\lhf$ group is also an $\lhg$ group, but it is unknown in general whether there exist any $\lhg$ groups which are not in $\lhf$. However, it has been conjectured by Bahlekeh, Dembegioti, and Talelli \cite[Conjecture~3.5]{ADTgor} that groups of finite Gorenstein cohomological dimension are in $\textup{H}_1\mathfrak{F}$. The truth of this conjecture would imply that the classes $\lhg$ and $\lhf$ coincide, although it is still open to the best of our knowledge. 
\par
We also have that every group of type $\Phi_k$ \cite{talphi} is an $\textup{LH}\mathfrak{G}_{AC}$, and hence an $\lhg$, group. Mazza and Symonds \cite{MSstabcat} construct a stable category for groups of type $\Phi_k$ and as one would hope it agrees with our stable category in this case. 
\par 
The Mazza-Symonds stable category can be constructed by formally inverting the syzygy functor via the complete cohomology of Benson and Carlson \cite{BCtate}. As mentioned in the introduction to \cite{MSstabcat}, this is the largest quotient category on which the syzygy functor is well defined and invertible. For more general $\lhg$ groups, this construction does not generalise very well. It is not even clear if the resulting category even has infinite coproducts, for example. 
\par 
The model category we construct is similar to that constructed by Benson \cite{bensoninf} and indeed for $\lhf$ groups they will have equivalent homotopy categories, as we show in \Cref{benstmodeq}. However, we have no restriction on the fibrant objects, which allows us to have more control over the behaviour of the category. This is the key, for example, in showing that the stable category is a tensor-triangulated category.
\par 
There are implications for the Gorenstein homological algebra of the hierarchically defined groups. It is well known that if $k$ is a field and $G$ is a finite group, then the projective modules and injective modules coincide. One consequence of this fact is that the Gorenstein projective and Gorenstein injective modules also coincide. For infinite groups, this will no longer be true, but we show that a weaker statement holds for $\lhg$ groups. In particular, we show that these groups are virtually Gorenstein in the sense that $\mathcal{GP}^{\perp} = {}^{\perp}\mathcal{GI}$; note that this property has been investigated over Artin algebras by Beligiannis and Krause \cite{BelKrause}. 
\begin{theorem}(See \Cref{eqperpsd}) 
    Suppose $k$ is a commutative ring of finite global dimension and $G$ is an $\textup{LH}\mathfrak{G}$ group. Then there is an equivalence of triangulated categories $\gp(kG) \simeq \gi(kG)$.
\end{theorem}
Bravo, Gillespie, and Hovey \cite{Bravo2014TheSM} introduced the Gorenstein AC-projective modules in their work to generalise the stable module category to arbitrary rings, where these modules are shown to always sit in the left hand side of a complete cotorsion pair, and hence have good approximation properties. \v{S}aroch and \v{S}t'ovíček \cite{stovsarmodel} defined the projectively coresolved Gorenstein flat modules and again show these always form the left hand side of a complete cotorsion pair. 
\par 
For any ring, there are inclusions $\mathcal{GP}_{AC} \subseteq \mathcal{PGF} \subseteq \mathcal{GP}$ by definition but it is unknown in general if these are ever strict. For $\textup{LH}\mathfrak{G}_{AC}$ groups, we show that these are indeed equalities. In particular, this implies that every Gorenstein projective is Gorenstein flat; it is another open question whether this is true in general.
\begin{theorem}(See \Cref{gpeqgpac} and \Cref{giegiac}) 
Suppose $k$ is a commutative noetherian ring of finite global dimension and $G$ is an $\textup{LH}\mathfrak{G}_{AC}$ group. Then the classes of Gorenstein projectives, Gorenstein AC-projectives, and projectively coresolved Gorenstein flat modules coincide. Dually, the class of Gorenstein injectives equals the class of Gorenstein AC-injectives. 
\end{theorem}
Finally, we investigate the compact objects in the stable category. For finite groups over a field, it is well known that these correspond to those modules which are stably isomorphic to a finitely generated module. In analogy with this, we show that the compact objects are those which are stable summands of $\textup{FP}_{\infty}$ Gorenstein projective modules in \Cref{thicksubcateescom}. In fact, we give another characterisation of the compact objects in \Cref{idcomps} and, using this, we prove the following; this is the Gorenstein equivalent of the fact that finitely generated projectives are finitely presented. 
\begin{theorem}(See \Cref{cute}) 
    Suppose $k$ is a commutative ring of finite global dimension and $G$ is an $\textup{LH}\mathfrak{G}$ group. Then every finitely generated Gorenstein projective is $\textup{FP}_{\infty}$, and furthermore has a complete resolution which is finitely generated in each degree. 
\end{theorem}
We show in \Cref{dcohco} that this has the consequence that groups of finite Gorenstein cohomological dimension are $d$-coherent for some $d > 0$, in the language of \cite{BRAVOdcoh}. This generalises the fact that $\mathbb{Z}G$ is noetherian (i.e. $0$-coherent) for $G$ a finite group.
\par
We see here another difference with finite groups: the tensor unit, which is the Gorenstein projective special precover of the trivial representation $k$, is not necessarily a compact object. We study for which groups it is compact and relate this to finiteness conditions on models for the classifying space for proper actions $\underline{E}G$, as well as finiteness conditions on the Eilenberg-Maclane space of the group. This builds on \cite[Theorem~A]{finconCEKT}.
\subsection*{Acknowledgements}
I would like to thank my supervisor Peter Symonds for his support and encouragement.
\section{Stable category}
We will construct our stable category as the homotopy category of a suitable model structure on $\textup{Mod}(kG)$. In order to construct this model structure, we will use the fundamental result of Hovey \cite[Theorem~2.2]{hovmod} showing a correspondence between model structures and certain cotorsion pairs. \par 
We begin by recalling the definitions we require. For any class of modules $\mathcal{S}$ we let 
\[\mathcal{S}^{\perp} := \{M \in \textup{Mod}(kG)\ |\ \textup{Ext}^1_{kG}(S,M) = 0 \textup{ for all } S \in \mathcal{S}\}\]
\[{}^{\perp}\mathcal{S} := \{M \in \textup{Mod}(kG)\ |\ \textup{Ext}^1_{kG}(M,S) = 0 \textup{ for all } S \in \mathcal{S}\}\]
Given two classes of modules $\mathcal{X}$ and $\mathcal{Y}$, we say $(\mathcal{X},\mathcal{Y})$ is a cotorsion pair if $\mathcal{X} = {}^{\perp}\mathcal{Y}$ and $\mathcal{X}^{\perp} = \mathcal{Y}$. We say a cotorsion pair $(\mathcal{X},\mathcal{Y})$ is complete if for any $kG$-module $M$ we have short exact sequences 
\[0 \to Y \to X \to M \to 0 \textup{ and } 0 \to M \to Y' \to X' \to 0\]
such that $X,X' \in \mathcal{X}$ and $Y,Y' \in \mathcal{Y}$. In this situation, we say that $X$ is an $\mathcal{X}$ special precover of $M$, and that $Y'$ is a $\mathcal{Y}$ special preenvelope of $M$. If these short exact sequences can be chosen functorially, we say $(\mathcal{X},\mathcal{Y})$ is functorially complete. 
\par We say a cotorsion pair is cogenerated by a set if there exists a set $\mathcal{S}$ (and not a proper class) such that $\mathcal{Y} = \mathcal{S}^{\perp}$; note that some sources refer to this as being generated by a set, but we follow Hovey's terminology in \cite{hovmod}. The importance of being cogenerated by a set lies in the following result of Eklof and Trlifaj.
\begin{theorem}\label{EkTrl}\cite[Theorem~10]{EkTrl} 
    Every cotorsion pair cogenerated by a set is complete.
\end{theorem}
\par 
Furthermore, a cotorsion pair is hereditary if $\textup{Ext}^i_{kG}(X,Y) = 0$ for all $i > 0$, $X \in \mathcal{X}$, and $Y \in \mathcal{Y}$. Finally, we call a class of modules $\mathcal{C}$ thick if it is closed under summands and for every short exact sequence $0 \to X \to Y \to Z \to 0$, if two of the modules are in $\mathcal{C}$ so is the third.
\par 
We can now state Hovey's theorem as follows; in fact the converse also holds but we will not use that here.
\begin{theorem}\label{HOVEYCORRESPONDONCE} \cite[Theorem~2.2]{hovmod} 
    Let $\mathcal{A}$ be an abelian category. Suppose $\mathcal{C},\mathcal{F},$ and $\mathcal{W}$ are classes of objects such that $\mathcal{W}$ is thick, and we have two complete cotorsion pairs $(\mathcal{C},\mathcal{W} \cap \mathcal{F})$ and $(\mathcal{C} \cap \mathcal{W}, \mathcal{F})$. Then there is a unique model structure on $\mathcal{A}$ such that $\mathcal{C}, \mathcal{F},$ and $\mathcal{W}$ are the class of cofibrant, fibrant, and trivial objects respectively. 
\end{theorem}
We now recall the required definitions from Gorenstein homological algebra. In order to state these, we first recall the following definitions from \cite{Bravo2014TheSM}. 
\begin{definition}
    A module $A$ is absolutely clean if $\textup{Ext}^1_{kG}(X,A) = 0$ for all modules $X$ of type $\textup{FP}_{\infty}$. Similarly, $L$ is called level if $\textup{Tor}_1^{kG}(X,L) = 0$ for all $X$ of type $\textup{FP}_{\infty}$.
\end{definition}
For example, if $kG$ is coherent, then the absolutely clean modules coincide with the injectives and the level modules are exactly the flat modules. 
\par 
We can now define the following types of modules. 
\begin{definition}\label{gorensteindef}
Let $M$ be a kernel in an acyclic complex $C_*$. We say $M$ is:
\begin{enumerate}[label=(\roman*)]
    \item Gorenstein projective if $C_*$ is a complex of projectives and $\textup{Hom}_{kG}(C_*,P)$ is acyclic for any projective $P$
    \item Gorenstein injective if $C_*$ is a complex of injectives and $\textup{Hom}_{kG}(I,C_*)$ is acyclic for any injective $I$
    \item Gorenstein AC-projective $C_*$ is a complex of projectives and $\textup{Hom}_{kG}(C_*,L)$ is acyclic for any level module $L$
    \item Gorenstein AC-injective if $C_*$ is a complex of injectives and $\textup{Hom}_{kG}(A,C_*)$ is acyclic for any absolutely clean module $A$
    \item Gorenstein flat if $C_*$ is a complex of flat modules and $I \otimes_{kG} C_*$ is acyclic for any injective (right) module $I$
    \item projectively coresolved Gorenstein flat if $M$ is Gorenstein flat and $C_*$ is a complex of projectives
\end{enumerate}
For a group $G$ we will use $\mathcal{GP}_G$ to denote the class of all Gorenstein projective $kG$-modules; sometimes we will omit the subscript $G$ if it is clear from context. 
\par 
Similarly, we will use $\mathcal{GI},\mathcal{GP}_{AC},\mathcal{GI}_{AC},\mathcal{GF},$ and $\mathcal{PGF}$ respectively, for the other classes of modules in the above definition.
\end{definition}
A complex such as in $(i)$ will be called a totally acyclic complex of projectives. 
\par 
We also need the notion of dimensions over the above types of modules. 
\begin{definition}
    We define the Gorenstein projective dimension of a module $M$ over $kG$, denoted $\textup{GPD}(M)$, to be the minimum integer $n$ such that there exists an exact sequence $0 \to G_n \to \dots \to G_0 \to M \to 0$ such that $G_i$ is a Gorenstein projective $kG$-module for each $0 \leq i \leq n$. 
    The global Gorenstein projective dimension is then the supremum over all $kG$-modules of $\textup{GPD}(M)$ and will be denoted $\textup{glGPD}(kG)$.
    \par 
    The (global) Gorenstein AC-projective dimension,(global) Gorenstein injective dimension and (global) Gorenstein AC-injective dimension all have similar definitions and notations. 
\end{definition}
We will use the notation $\mathfrak{G}$ (resp. $\mathfrak{G}_{AC})$ to denote the class of groups of finite global Gorenstein projective dimension (resp. finite global Gorenstein AC-projective dimension). 
\par 
We also fix some notation for later use.
\begin{notation}
    Let $H \leq G$ be a subgroup, $M$ a $kH$-module and $N$ a $kG$-module. We will write $N{\downarrow}_H^G$ to be the restriction of $N$. Similarly, we will write $M{\uparrow}_H^G$ and $M{\Uparrow}_H^G$ to denote induction and coinduction respectively.
\end{notation}
\begin{notation}
    We let $\underline{\textup{Hom}}_{kG}(M,N)$ be the additive quotient of $\textup{Hom}_{kG}(M,N)$ by the ideal of all maps which factor through a projective module. Similarly, we define $\overline{\textup{Hom}}_{kG}(M,N)$ to be the quotient by the ideal of all maps which factor through an injective module. 
    \par 
    The categories $\gp(kG)$ and $\gi(kG)$ are then defined in the obvious way.
\end{notation}
Finally, we recall Kropholler's hierarchically defined groups from \cite{krop}. 
\begin{definition}
    Let $\mathfrak{X}$ be a class of groups and let $\textup{H}_0\mathfrak{X} = \mathfrak{X}$. For every ordinal $\alpha > 0$, we define $\textup{H}_{\alpha}\mathfrak{X}$ inductively. If $\alpha$ is a successor ordinal, we let $\textup{H}_{\alpha}\mathfrak{X}$ be the class of groups which admit a cellular action on a finite dimensional contractible cell complex such that each stabiliser is in $\textup{H}_{\beta}\mathfrak{X}$ for some $\beta < \alpha$. For limit ordinals $\alpha$ let $H_{\alpha}\mathfrak{X} = \bigcup\limits_{\beta<\alpha}H_{\beta}\mathfrak{X}$. We define $\textup{H}\mathfrak{X}$ to be the union of $H_{\alpha}\mathfrak{X}$ over all ordinals $\alpha$. Finally, let $\lhx$ be the class of groups such that every finite subset of elements is contained in an $\textup{H}\mathfrak{X}$ subgroup.
\end{definition}
\subsection{Groups of finite Gorenstein cohomological dimension}\label{gcdse}
We will construct our stable category for $\textup{LH}\mathfrak{G}$ groups. We begin by investigating this class $\mathfrak{G}$, focusing on the properties we require for our later constructions. Throughout we always take $k$ to be a commutative ring of finite global dimension. 
\par 
We introduce the following standard terminology. 
\begin{definition}
    Let $k$ be a commutative ring and $G$ a group. The Gorenstein cohomological dimension of $G$ over $k$, denoted $\textup{Gcd}_k(G)$, is the Gorenstein projective dimension of the trivial module $k$. 
\end{definition}
We also have the following two cohomological invariants. 
\begin{definition}
    We define $\textup{silp}(kG)$ to be the supremum of the injective dimensions of the projective $kG$-modules. Similarly, $\textup{spli}(kG)$ is the supremum of the projective dimensions of the injective $kG$-modules. 
\end{definition}
We then have the following equivalent descriptions of groups of finite Gorenstein cohomological dimension.
\begin{theorem}\label{ETThe}
    Suppose $k$ is a commutative ring of finite global dimension and $G$ a group. Then the following are equivalent. 
    \begin{enumerate}[label=(\roman*)]
        \item $\textup{Gcd}_k(G) < \infty$
        \item $\textup{silp}(kG) = \textup{spli}(kG) < \infty$  
        \item Every $kG$-module has finite Gorenstein projective dimension
        \item $\textup{glGPD}(kG) < \infty$
    \end{enumerate}
\end{theorem}
\begin{proof}
    The equivalence $(i) \iff (ii) \iff (iii)$ are \cite[Theorem~1.7]{ETgcd}. The equivalence of these with $(iv)$ then follows from \cite[Corollary~1.5]{ETgcd}.
\end{proof}
Using this, we can show that being Gorenstein projective is a subgroup closed property as long as the subgroup has finite Gorenstein cohomological dimension. 
\begin{lemma}\label{allsnf}
    Suppose $G$ is a group and $M$ is a Gorenstein projective $kG$-module. Then $M{\downarrow}_H^G$ is Gorenstein projective if $\textup{Gcd}_k(H) < \infty$.
    \par 
    Similarly, if $N$ is a Gorenstein injective $kG$-module, then so is $M{\downarrow}_H^G$ as long as $\textup{Gcd}_k(H) < \infty$.
\end{lemma}
\begin{proof}
    It suffices to show that every exact complex of $kH$-projectives is totally acyclic. However, every $kH$-projective has finite injective dimension from \Cref{ETThe} and so the result follows by an induction on the injective dimenison cf. \cite[Lemma~3.14]{MSstabcat}.
    \par 
    The statement about Gorenstein injectives follows similarly, using that every $kH$-injective has finite projective dimension by \Cref{ETThe}.
\end{proof}

\begin{proposition}\label{fpdext}
    Suppose $\textup{Gcd}_k(G) < \infty$ and let $M$ be a $kG$-module. Then $M \in \mathcal{GP}^{\perp}$ if and only if $M$ has finite projective dimension. Furthermore, $M \in {}^{\perp}\mathcal{GI}$ if and only if $M$ has finite injective dimension.
\end{proposition}
\begin{proof}
Suppose $M$ has finite projective dimension. It is shown in  \cite[Theorem~2.20]{HOLM2004167} that this implies that $\textup{Ext}^1_{kG}(N,M) = 0$ for all Gorenstein projectives $N$. 
\par 
On the other hand, suppose $M \in \mathcal{GP}^{\perp}$. Since $\textup{Gcd}_k(G) < \infty$, we know $\textup{glGPD}(kG) < \infty$ and hence some syzygy $\Omega^d(M)$ must be Gorenstein projective. By \Cref{thicle} we know that $\mathcal{GP}^{\perp}$ is thick and so $\Omega^d(M) \in \mathcal{GP}^{\perp}$. However, $\mathcal{GP} \cap \mathcal{GP}^{\perp}$ coincides with the class of projective modules and hence $M$ has finite projective dimension.
\par 
The dual statement follows similarly, using \cite[Theorem~4.2]{em22} to see that every module has finite Gorenstein injective dimension in our situation.
\end{proof}
\begin{corollary}\label{perpequal}
    Suppose $G$ is a group with $\textup{Gcd}_k(G) < \infty$. Then $\mathcal{GP}^{\perp} = {}^{\perp}\mathcal{GI}$. In particular, $(\mathcal{GP},\mathcal{GP}^{\perp})$ is a cotorsion pair cogenerated by a set.
\end{corollary}
\begin{proof}
    In view of \Cref{fpdext} it suffices to show that a module has finite projective dimension if and only if it has finite injective dimension. From \Cref{ETThe} we know that $\textup{glGPD}(kG) < \infty$.  Then this is exactly the statement of \cite[Corollary~4.3]{em22}. 
    \par 
    To show $(\mathcal{GP},\mathcal{GP}^{\perp})$ is a cotorsion pair, we only need to show that ${}^{\perp}(\mathcal{GP}^{\perp}) \subseteq \mathcal{GP}$. Given that $\mathcal{GP}^{\perp}$ consists exactly of the modules of finite projective dimension and that every $kG$-module has finite Gorenstein projective dimension, this in turn follows immediately from \cite[Theorem~2.20]{HOLM2004167}. Alternatively, we could use \cite[Corollary~3.4]{cortésizurdiaga2023cotorsion} which shows that $(\mathcal{GP},\mathcal{GP}^{\perp})$ is a cotorsion pair over any ring. 
    \par 
    The statement about $(\mathcal{GP},\mathcal{GP}^{\perp})$ being cogenerated by a set follows from \Cref{gpandpgf} ahead.
\end{proof}
Hovey's \Cref{HOVEYCORRESPONDONCE} now shows that there is a model structure on $\textup{Mod}(kG)$ such that the Gorenstein projectives are the cofibrant objects, see \Cref{ourbigmodel} for the details if necessary. We will refer to this model structure as the Gorenstein projective model structure.  
\par 
We now show that the Gorenstein projective model structure is monoidal in the sense of Hovey \cite[Definition~4.2.6]{hovey2007model} for every group $G$ of finite Gorenstein cohomological dimension. This means that the following two statements hold. 
\begin{enumerate}
    \item \textit{(pushout-product axiom)} For any cofibrations $f: U \to V$ and $g: W \to X$ the pushout product $V\otimes W \bigsqcup\limits_{U \otimes W} U \otimes X \to V \otimes X$ is a cofibration, which is trivial if either $f$ or $g$ is trivial.
    \item \textit{(unit axiom)} Let $\mathcal{G}(\mathbbm{1}) \to \mathbbm{1}$ be the cofibrant replacement of the tensor unit $\mathbbm{1}$. Then the natural map $\mathcal{G}(\mathbbm{1}) \otimes X \to X$ is a weak equivalence for all cofibrant $X$.
\end{enumerate}
The reason for this definition stems from \cite[Theorem~4.3.2]{hovey2007model}, which shows that the homotopy category of a monoidal model category is itself a monoidal category, with the monoidal structure being given by the derived tensor product. 
\begin{remark}
The unit axiom is stated for a particular choice of (functorial) cofibrant replacement. However, if it holds for one choice then it holds for any cofibrant replacement. Indeed, consider two Gorenstein projective special precovers of $k$ as follows: 
\[0 \to K \to Q \to k \to 0 \textup{ and } 0 \to K' \to Q' \to k \to 0\]
Assume that the unit axiom holds for the cofibrant replacement $Q \to k$. 
Using that $K' \in \mathcal{GP}^{\perp}$ and $Q \in \mathcal{GP}$ we find that the map $Q \to k$ lifts through the identity to a map $Q \to Q'$. Necessarily, this map is a weak equivalence. The pushout product axiom ensures that the tensor is a left Quillen functor and so by Ken Brown's Lemma \cite[Lemma~1.1.12]{hovey2007model}, tensoring with a Gorenstein projective preserves weak equivalences between cofibrant objects. Therefore, we have the following commutative square. 
\[\begin{tikzcd}
    Q \otimes X \arrow{r}\arrow{d} & k \otimes X \arrow[equal]{d}\\ Q'\otimes X \arrow{r} & k \otimes X
\end{tikzcd}\]
The left vertical arrow and top horizontal arrow are weak equivalences and hence by the 2-out-of-3 property for weak equivalences, it follows that so is the bottom horizontal arrow.
\end{remark}
We use this in what follows, when we show the unit axiom holds for groups of finite Gorenstein cohomological dimension. 
\begin{proposition}\label{basemono}
    Suppose $G$ is a group with $\textup{Gcd}_k(G)< \infty$. Then the Gorenstein projective model structure is monoidal.
\end{proposition}
\begin{proof}
To show the pushout-product axiom we wish to apply \cite[Theorem~7.2]{hovmod}. Translating this to our context we need to show the following, noting that we have replaced hypothesis $(a)$ of \cite[Theorem~7.2]{hovmod} by the statement following \cite[Lemma~7.3]{hovmod}. 
\begin{enumerate}[label=(\roman*)]
        \item Every Gorenstein projective module is flat over $k$ and for every $kG$-module $M$ there is a surjective map $X \to M$ with $X$ Gorenstein projective
        \item If $X,Y$ are Gorenstein projective, so is $X \otimes Y$
        \item If $X,Y$ are Gorenstein projective, and at least one of them is projective then $X \otimes Y$ is projective 
\end{enumerate}
We also need to show the unit axiom, which takes the following form in our context. 
\begin{enumerate}[resume*]
\item For any Gorenstein projective $X$ the natural map $\mathcal{G}(k)\otimes X \to k \otimes X$ is a weak equivalence, where $\mathcal{G}(k) \to k$ is a Gorenstein projective special precover of $k$
\end{enumerate}
Since $k$ has finite global dimension, every Gorenstein projective is projective over $k$. We also know that Gorenstein projective special precovers exist; this follows by \Cref{EkTrl} and \Cref{perpequal}. Altogether, we see that $(i)$ holds. 
\par 
We now show that $(ii)$ holds. Suppose that $X$ and $Y$ are Gorenstein projective. Then $X$ is the kernel in a totally acyclic complex of projectives $P_*$. Hence, $X \otimes Y$ is the kernel in $P_* \otimes Y$. This is an exact complex of projectives and hence, as we showed in the proof of \Cref{allsnf}, it must be totally acyclic. It follows that $X \otimes Y$ is Gorenstein projective. 
    \par 
The statement of $(iii)$ is shown in \cite[Lemma~4.2]{MSstabcat}. 
\par 
    Now, we show the unit axiom, which is part $(iv)$. Consider the special precover of $k$, which is a short exact sequence $0 \to K \to \mathcal{G}(k) \to k \to 0$ such that $K \in \mathcal{GP}^{\perp}$. From \Cref{fpdext} this implies that $K$ has finite projective dimension and therefore, since $X$ is projective over $k$, we see that $K \otimes X$ must also have finite projective dimension and hence $K \otimes X \in \mathcal{GP}^{\perp}$. 
    \par 
    Tensoring the special precover of $k$ above with $X$ now implies that the map $\mathcal{G}(k) \otimes X \to k \otimes X$ is a surjection with kernel in $\mathcal{GP}^{\perp}$, i.e. it is a trivial fibration. In particular, it is a weak equivalence as we desired. 
\end{proof}
In fact, in this case the model structure satisfies the monoid axiom. As explained in \cite[Section~7]{hovmod}, this implies that the category of monoids form a model category, with the structure being induced by that of the Gorenstein projective model structure. The monoid axiom is actually satisfied for $\lhg$ groups and so we prove this now, although we have not yet proved the existence of the model structure in general, see \Cref{ttcatth}. 
\begin{proposition}
    The Gorenstein projective model structure satisfies the monoid axiom for every $\textup{LH}\mathfrak{G}$ group $G$.
\end{proposition}
\begin{proof}
    By \cite[Theorem~7.4]{hovmod} it suffices to prove the following two statements. 
    \begin{enumerate}[label=(\roman*)]
        \item If $X \in \mathcal{GP} \cap \mathcal{GP}^{\perp}$ and $Y$ is arbitrary, then $X \otimes Y \in \mathcal{GP}^{\perp}$
        \item $\mathcal{GP}^{\perp}$ is closed under transfinite compositions of pure monomorphisms
    \end{enumerate}
    We prove $(i)$ first, and begin by noting that necessarily $X$ is projective. It follows from \cite[Lemma~4.2]{MSstabcat} that $X \otimes Y$ has finite projective dimension. But it is straightforward to see that every module of finite projective dimension is contained in $\mathcal{GP}^{\perp}$, see e.g. \cite[Proposition~2.3]{HOLM2004167}. 
    \par 
    For part $(ii)$ we first recall from \cite[Lemma~6.2]{hovmod} that the left hand side of a cotorsion pair is closed under transfinite extensions. In particular ${}^{\perp}\mathcal{GI}$ must be closed under transfinite extensions. However, we know from \Cref{eqperpsd} and \Cref{perpequal} that for our groups we have equality $\mathcal{GP}^{\perp} = {}^{\perp}\mathcal{GI}$. Hence, $\mathcal{GP}^{\perp}$ is closed under transfinite compositions of pure monomorphisms, since it is closed under all transfinite extensions. 
\end{proof}
We now restrict ourselves to groups with $\textup{glGPD}_{AC}(kG) < \infty$. As we now show, these contain the class of groups of type $\Phi_k$, defined as follows. These are the groups considered by Mazza and Symonds in their work on the stable module category \cite{MSstabcat}.
\begin{definition}\cite{talphi}  We say that a group is of type $\Phi_k$ if any $kG$-module $M$ has finite projective dimension if and only if $M{\downarrow}_F^G$ has finite projective dimension for every finite subgroup $F \leq G$\end{definition}

\begin{proposition}
    Suppose $k$ is a commutative noetherian ring of finite global dimension and $G$ is a group of type $\Phi_k$. Then $\textup{glGPD}_{AC}(kG) < \infty$.
\end{proposition}
\begin{proof}
    It is shown in \cite[Theorem~3.9]{MSstabcat} that for groups of type $\Phi_k$, every module must have a complete resolution. This immediately implies that some syzygy of the trivial module $k$ must be Gorenstein projective, and hence $\textup{Gcd}_k(G)< \infty$. 
    \par 
    From \Cref{ETThe} we know that this is equivalent to $\textup{glGPD}(kG) < \infty$. We claim that $\mathcal{GP} = \mathcal{GP}_{AC}$; this implies the statement of the proposition.
    \par 
    In fact, we instead show that $\mathcal{GP}_{AC}^{\perp} \subseteq \mathcal{GP}^{\perp}$. To see why this suffices, note that this in turn implies that ${}^{\perp}(\mathcal{GP}^{\perp}) \subseteq {}^{\perp}(\mathcal{GP}_{AC}^{\perp})$. However, $(\mathcal{GP},\mathcal{GP}^{\perp})$ is a cotorsion pair and hence ${}^{\perp}(\mathcal{GP}^{\perp}) = \mathcal{GP}$. A similar statement is true for $\mathcal{GP}_{AC}$ and so altogether this shows that $\mathcal{GP} \subseteq \mathcal{GP}_{AC}$. By definition there is an inclusion $\mathcal{GP}_{AC} \subset \mathcal{GP}$ and so we conclude that $\mathcal{GP} = \mathcal{GP}_{AC}$ as desired. 
    \par 
    It remains to show that $\mathcal{GP}_{AC}^{\perp} \subseteq \mathcal{GP}^{\perp}$ as claimed. First, note that by \Cref{fpdext} we know that $\mathcal{GP}^{\perp}$ is exactly the class of modules of finite projective dimension. 
    \par 
    Suppose $M \in \mathcal{GP}_{AC}^{\perp}$. By what we have said, we must show that $M$ has finite projective dimension. Since $G$ is of type $\Phi$ it suffices to show that $M{\downarrow}_F^G$ has finite projective dimension for every finite subgroup $F \leq G$. Equivalently, again by \Cref{fpdext}, we may show that $M{\downarrow}_F^G \in \mathcal{GP}^{\perp}$. 
    \par The key point now is that, since $k$ is noetherian, the group ring $kF$ is noetherian for any finite group $F$. Hence, the level modules coincide with the flat modules. Furthermore, flat modules have finite projective dimension (see \cite[Corollary~4.3]{em22}). We conclude from this that $\mathcal{GP} = \mathcal{GP}_{AC}$ over $kF$. 
    \par 
    From this discussion, we see that we need to show that $M{\downarrow}_F^G \in \mathcal{GP}_{AC}^{\perp}$ for all finite subgroups $F \leq G$. This is proved as in \Cref{gpeqgpac} ahead.
\end{proof}
We remark that it is not known if every group with $\textup{glGPD}_{AC}(kG)$ must be of type $\Phi_k$. However, this does provide us with many natural examples of groups with $\textup{glGPD}_{AC}(kG)$. 
For example, groups of finite virtual cohomological dimension over $k$, and groups with a finite dimensional model for the classifying space for proper actions will have $\textup{glGPD}_{AC}(kG) < \infty$, see \cite[Corollary~2.6]{MSstabcat}.
\begin{definition}
    Following Gillespie \cite{gillespie_2019}, we say that a ring $R$ is AC-Gorenstein if for some $d > 0$, every right or left level $R$ module has absolutely clean coresolution of length at most $d$. This means for any level module $L$ there exists an exact sequence $0 \to L \to A_0 \to \dots \to A_d \to 0$ with each $A_i$ absolutely clean for $0 \leq i \leq d$.
\end{definition}
Since we are working over group rings, we can just focus on left $kG$-modules. 
\begin{lemma}\label{gpequalgacp}
Suppose $G$ is a group with $\textup{glGPD}_{AC}(kG) < \infty$. Then any level module has finite injective dimension and so $kG$ is an AC-Gorenstein ring. Furthermore, the Gorenstein projective $kG$-modules coincide with the Gorenstein AC-projective $kG$-modules. Dually, the Gorenstein injective modules coincide with the Gorenstein AC-injective modules.
\end{lemma}
\begin{proof}
    We first prove the claim that all level modules have finite injective dimension. Let $L$ be a level module and $M$ be arbitrary. There exists a resolution $0 \to G_n \to \dots G_0 \to M \to 0$ where each $G_i$ is Gorenstein AC-projective. By definition, we know that $\textup{Ext}^1_{kG}(G_i,L) = 0$ for each $0 \leq i \leq n$. By dimension shifting, we can now see that $\textup{Ext}^n_{kG}(M,L) \cong \textup{Ext}^1_{kG}(G_n,L) \cong 0$. It follows that $L$ has finite injective dimension. 
    \par Since injective modules are absolutely clean, it follows that $G$ is AC-Gorenstein once we know that these injective dimensions must be bounded. This follows from \cite[Corollary~4.3]{em22}.
    \par 
   Gillespie then constructs in \cite[Theorem~6.2]{gillespie_2019} a model structure with (one equivalent characterisation of) the trivial modules being those of finite level dimension. By what we have just proved, this is equivalent to a module being of finite injective dimension which in turn is equivalent, as noted in the proof of \Cref{perpequal}, to a module being of finite projective dimension. Therefore, we see from \Cref{fpdext} that these correspond to those modules in $\mathcal{GP}^{\perp}$. It follows that Gorenstein projective modules and Gorenstein AC-projective modules coincide, since these can both be described as ${}^{\perp}\mathcal{W}$, where $\mathcal{W}$ is the collection of modules of finite projective dimension.
   \par 
   The statement about Gorenstein injectives follows similarly, using Gillespie's dual model structure in \cite[Theorem~7.1]{gillespie_2019}.
\end{proof}
Having shown this, the following is then exactly the statement of \cite[Theorem~6.2]{gillespie_2019} interpreted in our context.
\begin{proposition}\label{prop:fpinfcompgen}
Suppose $k$ is a commutative ring of finite global dimension and $G$ is a group with $\textup{glGPD}_{AC}(kG) < \infty$. Then the cotorsion pair $(\mathcal{GP},\mathcal{GP}^{\perp})$ is cogenerated by the $kG$-modules of the form $\Omega^d(F)$ where $d = \textup{glGPD}_{AC}(kG)$ and $F$ runs over isomorphism classes of all $kG$-modules of type $\textup{FP}_{\infty}$. 
\end{proposition}
In particular, we will see in \Cref{compasecti} that this implies that the stable category for $\textup{LH}\mathfrak{G}_{AC}$ groups is compactly generated. 
\subsection{Constructing the stable category}
Throughout this section, we let $k$ be a commutative noetherian ring of finite global dimension. We will construct the stable category for $\textup{LH}\mathfrak{G}$ groups using Hovey's \Cref{HOVEYCORRESPONDONCE}. We work our way up to showing that we have two cotorsion pairs of the required form; we begin with a lemma. 
\begin{lemma}\label{thicle}
    The class $\mathcal{GP}^{\perp}$ is thick.
\end{lemma}
\begin{proof}
    First we claim that if $M \in \mathcal{GP}^{\perp}$ then $\textup{Ext}^i_{kG}(X,M) = 0$ for all $X \in \mathcal{GP}$. So, fix some Gorenstein projective $X$ and consider the short exact sequence $0 \to Y \to P \to X \to 0$ such that $P$ is projective and $Y$ is Gorenstein projective. Then $\textup{Ext}^2_{kG}(X,M) \cong \textup{Ext}^1_{kG}(Y,M) = 0$. An obvious iteration shows the claim. 
    \par 
    Now, the long exact sequence in $\textup{Ext}$ shows that $\mathcal{GP}^{\perp}$ is closed under cokernels of monomorphisms and extensions. Suppose that $0 \to M' \to M \to M''\to 0$ is a short exact sequence with $M,M'' \in \mathcal{GP}^{\perp}$ and let $X$ be Gorenstein projective. 
    \par 
    Using the above, the long exact sequence in $\textup{Ext}$ shows that $\textup{Ext}^i_{kG}(X,M') = 0$ for all $i \geq 2$; note that this is true for every Gorenstein projective $X$. Consider now the short exact sequence $0 \to X \to P' \to Y' \to 0$ with $P'$ projective and $Y'$ Gorenstein projective. Then $\textup{Ext}^1_{kG}(X,M') \cong \textup{Ext}^2_{kG}(Y',M')$, and the lattter is $0$ as we have just shown.
    \par 
    Finally, the fact that $\mathcal{GP}^{\perp}$ is closed under summands follows immediately.
\end{proof}
We also need the Gorenstein projectives to sit in the left hand side of a cotorsion pair; this has been shown over any ring by Cortés-Izurdiaga and \v{S}aroch, as follows. 
\begin{theorem}\cite[Corollary~3.4]{cortésizurdiaga2023cotorsion} 
    For any ring $(\mathcal{GP},\mathcal{GP}^{\perp})$ is a cotorsion pair.
\end{theorem}
Therefore, the difficulty lies in showing that the Gorenstein projective cotorsion pair is complete. We do this by showing that it is cogenerated by a set; before this, we need some preparation.
\begin{lemma}\label{use3}
    Let $H \leq G$. Suppose for every Gorenstein projective $kG$-module $M$ we have that $M{\downarrow}_H^G$ has finite Gorenstein projective dimension. Then $M{\downarrow}_H^G$ is in fact Gorenstein projective for every Gorenstein projective $kG$-module $M$.
\end{lemma}
\begin{proof}
    We first show that there exists some $n \geq 0$ such that $\textup{Gpd}_{kH}(\Omega^k(M){\downarrow}_H^G) \leq n$ for each $k \in \mathbb{Z}$. We note that since $M$ is Gorenstein projective, so is each $\Omega^k(M)$ and so, by assumption, $\Omega^k(M){\downarrow}_H^G$ has finite Gorenstein projective dimension. 
    \par
    \sloppy Suppose no such $n \geq 0$ exists, and so for each $m \geq 0$ we have some $k_m \in \mathbb{Z}$ such that $\textup{Gpd}_{kH}(\Omega^{k_m}(M){\downarrow}_H^G) \geq m$. Then, using \cite[Proposition~2.19]{HOLM2004167}, we know that  $\textup{Gpd}_{kH}(\bigoplus\limits_{m \in \mathbb{N}}\Omega^{k_m}(M){\downarrow}_H^G) = \sup\limits_{m \in \mathbb{N}}\textup{Gpd}_{kH}(\Omega^{k_m}(M){\downarrow}_H^G)$ which is infinite; this is a contradiction since $\bigoplus\limits_{m \in \mathbb{N}}\Omega^{k_m}(M)$ is Gorenstein projective and so by assumption has finite Gorenstein projective dimension on restriction to $H$.
    \par 
    Choose $n$ as above and consider the exact sequence 
    \[ 0 \to M \to P_0 \to \dots \to P_{-n} \to \Omega^{-n}(M) \to 0\]
    where each $P_i$ is projective. Restricting this to $H$ shows that $M$ must Gorenstein projective, since $\textup{Gpd}_{kH}(\Omega^{-n}(M{\downarrow}_H^G)) \leq n$.
\end{proof}
\begin{proposition}\label{resgp1} 
    Suppose $M$ is a Gorenstein projective $kG$-module. Then $M{\downarrow}_H^G$ is Gorenstein projective for every $\lhg$ subgroup $H \leq G$.
\end{proposition}
\begin{proof}
    First, suppose we have an $\textup{H}\mathfrak{G}$ subgroups $H \leq G$. We prove the statement by induction on the ordinal $\alpha$ such that $H$ is in $\textup{H}_{\alpha}\mathfrak{G}$. We have shown in \Cref{allsnf} that the statement is true for $\textup{H}_0\mathfrak{G} = \mathfrak{G}$ subgroups.
    \par 
    Suppose that it is true for all $\textup{H}_{\beta}\mathfrak{G}$ subgroups for some ordinal $\beta$, and let $H$ be an $\textup{H}_{\alpha}\mathfrak{G}$ subgroup of $G$ for some $\beta < \alpha$. Since $M$ is Gorenstein projective, and $k$ has finite global dimension, we know that $M$ is projective over $k$ and so $-\otimes M{\downarrow}_H^G$ is exact. Tensoring the cellular chain complex in the definition of $\textup{H}_{\alpha}\mathfrak{G}$ groups with $M{\downarrow}_H^G$ results in an exact complex of the following form 
    \[0 \to \bigoplus\limits_{i_k \in I_k}M{\downarrow}_{F_{i_k}}^G{\uparrow}_{F_{i_k}}^H \to \dots \to \bigoplus\limits_{i_0 \in I_0}M{\downarrow}_{F_{i_0}}^G{\uparrow}_{F_{i_0}}^H \to M{\downarrow}_H^G \to 0\]
    where each $F_{i_j} \leq H$ is an $\textup{H}_{\beta}\mathfrak{G}$ subgroup for $\beta < \alpha$. By the inductive hypothesis, each $M{\downarrow}_{F_{i_k}}^G{\uparrow}_{F_{i_k}}^H$ is Gorenstein projective and so $M{\downarrow}_H^G$ has finite Gorenstein projective dimension. This is true for any Gorenstein projective and so we conclude by \Cref{use3} that $M{\downarrow}_H^G$ is Gorenstein projective as desired.
    \par 
    We now let $H$ be an $\lhg$ subgroup, and we induct on the cardinality of $H$. If $H$ is countable, then it is an $\textup{H}\mathfrak{G}$ group and so we are done. Suppose, therefore, that $H$ is uncountable and write $H = \bigcup\limits_{i \in I}H_i$ as an ascending union such that each $H_i$ has cardinality strictly less than that of $H$. By induction, we know that $M{\downarrow}_{H_i}^G$ is Gorenstein projective for each $i \in I$. It is shown in \cite[Proposition~3.1]{ETgcd} that $\textup{Gpd}_{kH}(M{\downarrow}_H^G) \leq 1$. As this is true for every Gorenstein projective, the result now follows again by \Cref{use3}.
\end{proof}
We have the following variation on the above argument, which is a slight generalisation of \cite[Corollary~D]{demtalres}. 
\begin{proposition}\label{costabil}
    Let $G$ be an $\lhg$ group and consider an exact complex of projectives $P_*$. Then every kernel in $P_*$ is Gorenstein projective. In particular, acyclic complexes of projectives are totally acyclic.
\end{proposition}
\begin{proof}
    Let $K$ be the kernel in degree zero of $P_*$. To begin, we note that if $H \leq G$ is a $\mathfrak{G}$-subgroup, then $K{\downarrow}_H^G$ is Gorenstein projective. This is because acyclic complexes of projectives are automatically totally acyclic over $\mathfrak{G}$ groups; see the proof of \Cref{allsnf}. 
    \par 
    Now, the argument follows just as in \Cref{use3} and \Cref{resgp1}. Indeed, we only needed two properties of Gorenstein projectives for these. The first was the existence of an exact sequence $0 \to K \to P_0 \to \dots \to P_{-n} \to \Omega^{-n}(K) \to 0$ with $P_i$ all projectives; this certainly exists for us. The second was that Gorenstein projectives are projective over $k$. However, this again will hold for us since $k$ has finite global dimension. 
\end{proof}
\begin{theorem}\label{induc}
Suppose $G$ is an $\textup{LH}\mathfrak{G}$ group. Furthermore, suppose that $(\mathcal{GP}_{H},\mathcal{GP}_{H}^{\perp})$ is a cogenerated by a set $\mathcal{X}_{H}$ for each $H \in \mathfrak{G}$. Then $(\mathcal{GP}_G,\mathcal{GP}_G^{\perp})$ is a cogenerated by a set and hence complete.
\end{theorem}
\begin{proof}
We only need to show that $(\mathcal{GP}_G,\mathcal{GP}_G^{\perp})$ is cogenerated by a set; completeness then follows by \Cref{EkTrl}.
\par 
Note that for $\mathfrak{G}$ groups we have shown in \Cref{perpequal} that $(\mathcal{GP},\mathcal{GP}^{\perp})$ is cogenerated by a set. We first prove the claim for $\textup{H}\mathfrak{G}$ groups by transfinite induction. The base case is clearly true and so suppose that for some ordinal $\alpha$ it is true for $H_{\beta}\mathfrak{G}$ groups for all $\beta < \alpha$. 
    \par 
    Now, let $G$ be an $\textup{H}_{\alpha}\mathfrak{G}$ group. 
    We claim that $(\mathcal{GP}_G,\mathcal{GP}_G^{\perp})$ is cogenerated by the set $\bigcup\limits_H\mathcal{X}_H{\uparrow}_H^G$, where the union is over all $H_{\beta}\mathfrak{G}$ subgroups $H$ for $\beta < \alpha$. The fact that this is actually a set follows by the inductive hypothesis. 
    \par 
    Firstly, note that $X{\uparrow}_H^G$ is a Gorenstein projective $kG$-module if $X$ is a Gorenstein projective $kH$-module. Therefore, the inclusion $\mathcal{GP}^{\perp} \subseteq (\bigcup\limits_H\mathcal{X}_H{\uparrow}_H^G)^{\perp}$ follows immediately.
    \par 
    Suppose that $M$ is a $kG$-module such that $\textup{Ext}^1_{kG}(X{\uparrow}_H^G,M) = 0$ for any $X{\uparrow}_H^G \in \bigcup\limits_H\mathcal{X}_H{\uparrow}_H^G$ and let $Y$ be any Gorenstein projective $kG$-module. Note that for any $X \in \mathcal{X}_H$ we have $\textup{Ext}^1_{kH}(X,M{\downarrow}_H^G) = 0$. 
    \par Consider the chain complex given by tensoring the cellular chain complex in the definition of $\textup{H}\mathfrak{G}$ groups with $Y$ as follows. 
    \[0 \to \bigoplus\limits_{i_k \in I_k}Y{\downarrow}_{H_{i_k}}^G{\uparrow}_{H_{i_k}}^G \to \dots \to \bigoplus\limits_{i_0 \in I_0}Y{\downarrow}_{H_{i_0}}^G{\uparrow}_{H_{i_0}}^G \to Y \to 0\]
    By definition, each $H_{i_j}$ is a $\textup{H}_{\beta}\mathfrak{G}$ subgroup for some $\beta < \alpha$. 
    \par
    We claim that $\textup{Ext}^{k+1}_{kG}(Y,M) = 0$. To prove this, firstly note that by \Cref{resgp1} we know that $Y{\downarrow}_{H_{i_j}}^G$ is Gorenstein projective for any $0 \leq j \leq k$ and since $(\mathcal{GP}_{H_{i_j}},\mathcal{GP}_{H_{i_j}}^{\perp})$ is cogenerated by $\mathcal{X}_{H_{i_j}}$, the inductive hypothesis ensures that $\textup{Ext}^1_{kH_{i_j}}(Y{\downarrow}_{H_{i_j}}^G,M{\downarrow}_{H_{i_j}}) = 0$. Using the Eckmann-Shapiro isomorphism then shows that $\textup{Ext}^1_{kG}(Y{\downarrow}_{H_{i_j}}^G{\uparrow}_{H_{i_j}}^G,M) = 0$. 
    \par 
    In fact, from \Cref{thicle} we have that $\textup{Ext}^n_{kH_{i_j}}(Y{\downarrow}_{H_{i_j}}^G,M{\downarrow}_{H_{i_j}}) = 0$ for all $n \geq 1$ and so $\textup{Ext}^n_{kG}(Y{\downarrow}_{H_{i_j}}^G{\uparrow}_{H_{i_j}}^G,M) = 0$ for all $n \geq 1$. A dimension shifting argument then shows that $\textup{Ext}^{k+1}_{kG}(Y,M) = 0$ as claimed. Recall that $Y$ was an arbitrary Gorenstein projective and so this holds for any Gorenstein projective. 
    \par 
    Let $Z \in \mathcal{GP}_G$ and recall that we wish to prove $\textup{Ext}^1_{kG}(Z,M) = 0$. Since $Z$ is Gorenstein projective, we have an exact sequence of the form 
    \[ 0 \to Z \to P_0 \to \dots \to P_k \to Z' \to 0\]
    such that $Z'$ is also Gorenstein projective and each $P_i$ is projective for $0 \leq i \leq k$. Another dimension shifting argument proves that $\textup{Ext}^1_{kG}(Z,M) \cong \textup{Ext}^{k+1}_{kG}(Z',M)$, which we have just shown vanishes. This completes the proof that $\mathcal{GP}_G^{\perp} = (\bigcup\limits_H\mathcal{X}_H{\uparrow}_H^G)^{\perp}$ in this case. 
    \par 
    We now let $G$ be an $\textup{LH}\mathfrak{G}$ group and prove the statement by induction on the cardinality of $G$. Let $G = \bigcup\limits_{\alpha < \gamma}G_{\alpha}$, for some ordinal $\gamma$, such that each $G_{\alpha}$ has cardinality strictly smaller than $G$. By induction, we have a set $\mathcal{X}_{\alpha}$ such that $\mathcal{GP}_{G_{\alpha}}^{\perp} = \mathcal{X}_{\alpha}^{\perp}$ for each $\alpha < \gamma$. We claim that $(\mathcal{GP}_G,\mathcal{GP}_G^{\perp})$ is cogenerated by the set $\bigcup\limits_{\alpha < \gamma}\mathcal{X}_{\alpha}{\uparrow}_{G_{\alpha}}^G$.
    \par 
    As before, the inclusion $\mathcal{GP}_G^{\perp} \subseteq (\bigcup\limits_{\alpha < \gamma}\mathcal{X}_{\alpha}{\uparrow}_{G_{\alpha}}^G)^{\perp}$ follows immediately. Suppose, therefore, that $M \in (\bigcup\limits_{\alpha < \gamma}\mathcal{X}_{\alpha}{\uparrow}_{G_{\alpha}}^G)^{\perp}$ and take a Gorenstein projective $kG$-module $X$.
    \par 
    Suppose $\alpha \leq \beta \leq \gamma$ and consider the following commutative diagram. 
    \[\begin{tikzcd}
        K_{\alpha} \arrow{r}\arrow{d} & X{\downarrow}_{G_1}^G{\uparrow}_{G_1}^G \arrow{r} \arrow{d} & X{\downarrow}_{G_{\alpha}}^G{\uparrow}_{G_{\alpha}}^G \arrow{d}\\
        K_{\beta} \arrow{r} & X{\downarrow}_{G_1}^G{\uparrow}_{G_1}^G \arrow{r} & X{\downarrow}_{G_{\beta}}^G{\uparrow}_{G_{\beta}}^G
    \end{tikzcd}\]
    The Snake Lemma shows that the right vertical map is surjective and the left vertical map is injective and that $\textup{Ker}(X{\downarrow}_{G_{\alpha}}^G{\uparrow}_{G_{\alpha}}^G \to X{\downarrow}_{G_{\beta}}^G{\uparrow}_{G_{\beta}}^G) = \textup{Coker}(K_{\alpha} \to K_{\beta})$.
    \par 
    For any limit ordinal $\kappa$ we know that $G = \bigcup\limits_{\mu < \kappa}G_{\mu}$ and so $X = \varinjlim\limits_{\mu < \kappa}X{\downarrow}_{G_{\mu}}^G{\uparrow}_{G_{\mu}}^G$. Taking direct limits of the above system of short exact sequences shows that $\varinjlim\limits_{\mu < \kappa} K_{\mu} \cong K_{\kappa}$. In particular, we see that $\varinjlim\limits_{\alpha < \gamma}K_{\alpha} \cong K$, where $K$ is the kernel of the surjective map $X{\downarrow}_{G_1}^G{\uparrow}_{G_1}^G \to X$; this means that $K$ has a transfinite filtration by the $K_{\alpha}$.
    \par 
    Consider $\alpha$ such that $\alpha + 1 < \gamma$ and consider the short exact sequence 
    \[ 0 \to K_{\alpha +1}/K_{\alpha} \to X{\downarrow}_{G_{\alpha+1}}^G{\uparrow}_{G_{\alpha+1}}^G \to X{\downarrow}_{G_{\alpha}}^G{\uparrow}_{G_{\alpha}}^G \to 0\]
    By \Cref{resgp1} we know that $X{\downarrow}_{G_{\alpha}}^G$ is Gorenstein projective and hence, by the inductive hypothesis, $\textup{Ext}^1_{kG_{\alpha}}(X{\downarrow}_{G_{\alpha}}^G, M{\downarrow}_{G_{\alpha}}^G) = 0$. The same is true for $G_{\alpha + 1}$ and hence we see $\textup{Ext}^1_{kG}(K_{\alpha + 1}/K_{\alpha}, M ) = 0$.
    \par 
    Applying the Eklof Lemma \cite[Lemma~1]{EkTrl} shows that $\textup{Ext}^1_{kG}(K,M) = 0$. Combining this with the short exact sequence $K \to X{\downarrow}_{G_1}^G{\uparrow}_{G_1}^G \to X$ shows that $\textup{Ext}^2_{kG}(X,M) = 0$. As before, this is true for any Gorenstein projective $X$. 
    \par 
    Finally, choose any $Z \in \mathcal{GP}$ and consider the short exact sequence $0 \to Z \to P \to Z' \to 0$ where $P$ is projective and $Z'$ is Gorenstein projective. Then $\textup{Ext}^1_{kG}(Z,M) = \textup{Ext}^2_{kG}(Z',M) = 0$ and we have finished.
\end{proof}
In fact, the above argument shows that $(\mathcal{GP}_G,\mathcal{GP}_G^{\perp})$ is cogenerated by the set $\bigcup\limits_{H \in \mathfrak{G}}\mathcal{X}_H$. We use this in the following corollary.
\begin{corollary}\label{corfsub}
    We have $X \in \mathcal{GP}_G^{\perp}$ if and only if $X{\downarrow}_H^G \in \mathcal{GP}_H^{\perp}$ for all $H \in \mathfrak{G}$.
\end{corollary}
\begin{proof}
    We begin by showing the forward implication, and so take $X \in \mathcal{GP}_G^{\perp}$. Take $H \in \mathfrak{G}$ and Gorenstein projective $kH$-module $M$. Then $\textup{Ext}^1_{kH}(M,X{\downarrow}_H^G) \cong \textup{Ext}^1_{kG}(M{\uparrow}_H^G,X) \cong 0$. 
    \par 
    For the reverse implication we proceed, as usual, by transfinite induction, with the statement being true for $\textup{H}_0\mathfrak{G}$ groups. Now, let $G$ be an $\textup{H}_{\alpha}\mathfrak{G}$ group and suppose $X{\downarrow}_H^G \in \mathcal{GP}_H^{\perp}$ for all $H \in \mathfrak{G}$. By the inductive hypothesis, we in fact know that $X{\downarrow}_{H'}^G \in \mathcal{GP}_{H'}^{\perp}$ for all $H' \leq G$ with $H'$ an $\textup{H}_{\beta}\mathfrak{G}$ subgroup for $\beta < \alpha$. 
    \par 
    Therefore, for such an $H' \leq G$ we see that $\textup{Ext}^1_{kG}(Y{\uparrow}_{H'}^G,X) = 0$ for each Gorenstein projective $kH'$-module $Y$. We conclude from \Cref{induc} that $X \in \mathcal{GP}_{G}^{\perp}$. 
    \par 
    If $G$ is an $\textup{LH}\mathfrak{G}$ group, then as before we induct on the cardinality of $G$. Since countable $\textup{LH}\mathfrak{G}$ groups are necessarily in $\textup{H}\mathfrak{G}$, we are done in this case. If $G$ is uncountable, we can write $G = \bigcup\limits_{i \in I}G_{i}$ where each $G_i$ has cardinality strictly smaller than $G$. The inductive hypothesis implies that $X{\downarrow}_{G_i}^G \in \mathcal{GP}_{G_i}^{\perp}$ for each $i \in I$. We conclude, again using \Cref{induc}, that $X \in \mathcal{GP}_G^{\perp}$. 
\end{proof}
We record the following for future reference; it shows the essential properties of $\mathfrak{G}$ groups that we have used during the proofs of the previous results. We will use it in particular for $\lhf$ groups.
\begin{theorem}\label{corfsubfin}
    Suppose $\mathfrak{X}$ is a class of groups such that for each $H \in \mathfrak{X}$ the Gorenstein projective cotorsion pair is cogenerated by a set $\mathcal{X}_H$. Furthermore, assume that if $M$ is Gorenstein projective for some group $G$, then $M{\downarrow}_H^G$ is Gorenstein projective for any $\mathfrak{X}$-subgroup $H \leq G$. Then for any $\textup{LH}\mathfrak{X}$ group $G$, the cotorsion pair $(\mathcal{GP},\mathcal{GP}^{\perp})$ is cogenerated by $\bigcup\limits_{H \in \mathfrak{X}}\mathcal{X}_H$. Furthermore, $X \in \mathcal{GP}^{\perp}_G$ if and only if $X{\downarrow}_H^G \in \mathcal{GP}^{\perp}_H$ for all $H \in \mathfrak{X}$.
\end{theorem}
\begin{proof}
    This is proved just as \Cref{induc} and \Cref{corfsub}.
\end{proof}
Combining our work so far with \Cref{HOVEYCORRESPONDONCE} then gives us our desired model structure as follows. 
\begin{theorem}\label{ourbigmodel}
Suppose $G$ is an $\textup{LH}\mathfrak{G}$ group. Then there is a cofibrantly generated hereditary, abelian model structure on $\textup{Mod}(kG)$ such that every module is fibrant and the Gorenstein projectives are the cofibrants. The homotopy category is equivalent to the stable category of Gorenstein projectives $\gp(kG)$.
\end{theorem}
\begin{proof}
    We have shown that $\mathcal{GP}^{\perp}$ is thick in \Cref{thicle}. In order to apply Hovey's \Cref{HOVEYCORRESPONDONCE} we need to show that $(\mathcal{GP}\cap \mathcal{GP}^{\perp},\textup{Mod}(kG))$ and $(\mathcal{GP},\mathcal{GP}^{\perp})$ are complete cotorsion pairs. We have shown that this is true for the latter pair in \Cref{induc}. 
    \par 
    To show that $(\mathcal{GP}\cap \mathcal{GP}^{\perp},\textup{Mod}(kG))$ is a complete cotorsion pair, we first claim that $\mathcal{GP} \cap \mathcal{GP}^{\perp} = \mathcal{P}$, where $\mathcal{P}$ denotes the class of projective modules. For any $X \in \mathcal{GP} \cap \mathcal{GP}^{\perp}$, consider the short exact sequence $X \to P \to Y$ with $P$ projective and $Y \in \mathcal{GP}$. This splits by assumption, showing the claim. Now, it is clear that $(\mathcal{P},\textup{Mod}(kG))$ is a complete cotorsion pair. 
    \par 
    The fact the resulting model structure is cofibrantly generated follows from \cite[Lemma~6.7, Corollary~6.8]{hovmod}, since we have shown it is cogenerated by a set. 
    \par 
    Finally, the statement about the homotopy category is a fundamental property of model categories, see \cite[Theorem~2.6]{gilhermodcat} for a reference.
\end{proof}
We will use $\tacstab(kG)$ to denote the stable category when we do not wish to specify which model. Similarly, morphisms will be denoted by $\tachom_{kG}(-,-)$.
\begin{remark}\label{mapsrmk}
    This gives us the following explicit description of the maps. The (trivial) cofibrations are the injective homomorphisms with Gorenstein projective cokernel (with projective cokernel). The (trivial) fibrations are the surjective homomorphisms (with kernel in $\mathcal{GP}^{\perp}$). The weak equivalences are those which factor as a injective map with trivial cokernel followed by a surjective map with trivial kernel; this follows from \cite[Proposition~2.3]{gilhermodcat}.
\end{remark}
Recall that a cotorsion pair is said to be of finite (countable) type if it is cogenerated by a set of $\textup{FP}_{\infty}$ modules (countably presented modules). 
\begin{corollary}\label{fintypec}
     If $G$ is an $\textup{LH}\mathfrak{G}_{AC}$ group then the cotorsion pair $(\mathcal{GP},\mathcal{GP}^{\perp})$ is of finite type.
\end{corollary}
\begin{proof}
    During the proof of \Cref{induc} we have shown that the Gorenstein projective cotorsion pair is cogenerated by $\bigcup\limits_{H \in \mathfrak{G}_{AC}}\mathcal{X}_H{\uparrow}_H^G$ where $H$ runs over the $\mathfrak{G}_{AC}$-subgroups of $G$, and $\mathcal{X}_H$ is the cogenerating set for the cotorsion pair over $H$. 
    \par 
    We have shown in \Cref{prop:fpinfcompgen} that the Gorenstein projective cotorsion pair is of finite type for groups with $\textup{glGPD}_{AC}(kG) < \infty$ and the statement follows.
\end{proof} 
It is not clear if the Gorenstein projective cotorsion pair is of countable type for every group, but with an additional assumption we can show it holds for $\mathcal{PGF}$ modules.
\begin{proposition}
    Assume that being projectively coresolved Gorenstein flat is a subgroup closed property. Then $(\mathcal{PGF},\mathcal{PGF}^{\perp})$ is of countable type for every countable ring $k$ and group $G$. 
\end{proposition}
\begin{proof}
    We write $G$ as a filtered union of all its finitely generated subgroups $G = \bigcup\limits_{\alpha < \gamma}G_{\alpha}$. The argument in the proof of \Cref{induc} also works in this context to show that $\mathcal{PGF}^{\perp} = (\bigcup\limits_{\alpha < \gamma}\mathcal{X}_{\alpha}{\uparrow}_{G_{\alpha}}^G)^{\perp}$, where $\mathcal{X}_{\alpha}$ is the set such that $\mathcal{X}_{\alpha}^{\perp} = \mathcal{PGF}_{G_{\alpha}}^{\perp}$ for each $\alpha < \gamma$. It is important here that restriction preserves the property of being projectively coresolved Gorenstein flat.
    \par 
    We note that $kG_{\alpha}$ is $\aleph_0$-coherent for each finitely generated group $G_{\alpha}$ and countable ring $k$. Therefore, it is shown in \cite[Theorem~3.9]{stovsarmodel} that $\mathcal{X}_{\alpha}$ can be chosen to consist of countably presented modules. The result then follows. 
\end{proof}
As commented after \cite[Theorem~3.9]{stovsarmodel}, this identifies the class $\mathcal{PGF}$ as the subclass of Gorenstein projectives consisting of those which have a filtration by countably presented Gorenstein projectives. 
\section{Comparison with other stable categories}\label{compsection}
There are various other possibilities to take for the stable category which have been considered in the literature. We now investigate when these will be equivalent to the stable category we have just constructed. 
\subsection{Benson's model structure}
In \cite[Section~10]{bensoninf}, Benson constructs a model category for any group $G$ and $k$ a commutative noetherian ring. To state this, let $B$ denote the set of bounded functions from $G$ to $k$. This is made into a $kG$-module via the diagonal action. 
\begin{theorem}\label{benstmodsthm}\cite[Theorem~10.6]{bensoninf} 
    Let $\textup{Mod}_f(kG)$ be the category of all modules $M$ such that $B \otimes_k M$ has finite projective dimension. Then there is a model category structure on $\textup{Mod}_f(kG)$ such that each module in $\textup{Mod}_f(kG)$ is fibrant and the cofibrant modules are those $N$ such that $B \otimes_k N$ is projective. 
\end{theorem}
\begin{proposition}\label{benstmodeq}
    Suppose $k$ is a commutative noetherian ring of finite global dimension and $G$ is an $\lhf$ group. Then the homotopy category of Benson's model structure is equivalent to the stable category $\tacstab(kG)$.
\end{proposition}
\begin{proof}
    It is shown in \cite[Corollary~C]{demtalres} that for an $\lhf$ group over $\mathbb{Z}$ the cofibrant modules in Benson's model structure coincide with the Gorenstein projectives. The same proof works for any commutative ring of finite global dimension and so the statement follows. 
\end{proof}
\begin{remark}
    It would be interesting to know if there were a way to prove the above without \cite[Corollary~C]{demtalres}, and more in line with the method of \Cref{gpeqgpac}. However, it is not clear if induction preserves the cofibrant modules in Benson's model structure. We will however give an alternative proof in \Cref{ourcorC}.
\end{remark}
\subsection{Mazza-Symonds stable category}
Let $k$ be a commutative, noetherian ring of finite global dimension. Mazza and Symonds \cite{MSstabcat} define a stable category as follows. 
\begin{definition}\label{firstcomcodef}
    Let $\textup{Stab}(kG)$ be the category with objects given by all $kG$-modules and morphisms $\widehat{\textup{Hom}}_{kG}(M,N)$ given by complete cohomology, i.e. 
    \[\widehat{\textup{Hom}}_{kG}(M,N) = \varinjlim\limits_n\underline{\textup{Hom}}_{kG}(\Omega^n(M),\Omega^n(N))\]
\end{definition}
For groups of type $\Phi_k$, Mazza and Symonds show that this category is equivalent to the stable category of Gorenstein projectives, and hence the stable category $\tacstab(kG)$.
\begin{theorem}
    Let $G$ be a group. Then the following categories are equivalent. 
    \begin{enumerate}[label=(\roman*)]
        \item $\textup{Stab}(kG)$
        \item $D^b(\textup{Mod}(kG))/K^b(\textup{Proj}(kG))$
        \end{enumerate}
        If $G$ is a group of type $\Phi_k$ then these are also equivalent to 
        \begin{enumerate}[resume*]
        \item $\gp(kG)$
        \item The homotopy category of totally acyclic complexes of projectives
        \item $\tacstab(kG)$
    \end{enumerate}
\end{theorem}
\begin{proof}
    The equivalence of $(i)$ and $(ii)$ is shown by Beligiannis \cite[Corollary~3.9]{belstcat}. It is also well known that $(iii)$ and $(iv)$ are equivalent over any ring, for example this follows from \cite[Proposition~4.4.18]{krause_2021}. 
    \par 
    If $G$ is a group of type $\Phi$, then the equivalences not including $\tacstab(kG)$ are shown in \cite[Theorem~3.10]{MSstabcat}. The fact that all these categories are equivalent to $\tacstab(kG)$ for a group of type $\Phi$ now follows from \Cref{ourbigmodel}.
\end{proof}
We now show that for groups which are not of type $\Phi$, the equivalences of the above theorem do not have to hold. In particular, the Mazza-Symonds stable category will not be equivalent to the stable category $\tacstab(kG)$ we have constructed; recall from \Cref{ourbigmodel} that this is equivalent to $\gp(kG)$.
\begin{proposition}
    Suppose $G$ is an $\lhf$ group which is not of type $\Phi$. Then the Mazza-Symonds stable category $\textup{Stab}(kG)$ is not equivalent to $\tacstab(kG)$.
\end{proposition}
\begin{proof}
     \sloppy Kropholler shows in \cite[Theorem~4.2]{krop} that for a $kG$-module $M$ we have that $\widehat{\textup{Hom}}_{kG}(M,M) = 0$ if and only if $M$ has finite projective dimension. In particular, $M \cong 0$ in $\textup{Stab}(kG)$ if and only if $M$ has finite projective dimension. 
    \par 
    However, we have shown in \Cref{corfsubfin} that $M \cong 0$ in $\tacstab(kG)$ if and only if $M{\downarrow}_H^G \cong 0$ in $\tacstab(kH)$ for each finite subgroup $H \leq G$. Since $G$ is not of type $\Phi$, there exists a $kG$-module $N$ such that $N{\downarrow}_H^G$ has finite projective dimension for every finite subgroup $H \leq G$, but $N$ has infinite projective dimension. In particular, $N \cong 0$ in $\tacstab(kG)$ but $N$ is not isomorphic to zero in $\textup{Stab}(kG)$. 
\end{proof}
\subsection{Gorenstein injective model structure}
Throughout this section we continue to assume that $k$ is a commutative ring of finite global dimension. 
\par 
Suppose now that $G$ is an $\textup{LH}\mathfrak{G}$ group. We show that the homotopy category of the model structure of \Cref{ourbigmodel} is equivalent to the homotopy category of the Gorenstein injective model structure which is constructed by Šaroch and Št'ovíček over any ring in \cite[Section~4]{stovsarmodel}. The following is proved in \cite{stovsarmodel} over any ring $R$, but we only state it for group rings.
\begin{theorem}\label{injmodelstr} \cite[Theorem~4.6]{stovsarmodel} 
    Let $\mathcal{GI}$ be the class of Gorenstein injective $kG$-modules. Then there is an hereditary, perfect cotorsion pair $({}^{\perp}\mathcal{GI},\mathcal{GI})$ and hence a model category structure on $\textup{Mod}(kG)$ such that every module is cofibrant, the Gorenstein injectives are the fibrant objects and ${}^{\perp}\mathcal{GI}$ is the class of trivial objects.
\end{theorem}
Note that the existence of the model category structure in the above follows from \cite[Theorem~2.2]{gilhermodcat} and is not stated explicitly in the statement of \cite[Theorem~4.6]{stovsarmodel}, although it is mentioned afterwards. 
\par
In order to show that this is equivalent to the Gorenstein projective model category we constructed earlier, we first identify the trivial modules of the injective model structure. 
\begin{proposition}\label{dugi}
    Suppose $G$ is a group and $M$ is a Gorenstein injective $kG$-module. Then $M{\downarrow}_H^G$ is a Gorenstein injective $kH$-module for every $\textup{H}\mathfrak{G}$ subgroup $H \leq G$.
\end{proposition}
\begin{proof}
    This is dual to \Cref{resgp1}.
\end{proof}
It's worthwhile noting that we do not claim the above is true for $\textup{LH}\mathfrak{X}$ subgroups; this is because the result of Emmanouil and Talelli \cite[Proposition~3.1]{ETgcd} on the Gorenstein projective dimension over filtered unions of groups is not clearly dualisable. However, it is sufficient for our purposes to only hold for $\textup{H}\mathfrak{G}$ subgroups. 
\par 
 The following is a partial dual to \cite[Theorem~5.7]{bensoninf1}.
\begin{lemma}\label{dugi2}
    Suppose $G$ is an $\textup{H}\mathfrak{G}$ group and $M$ is a Gorenstein injective $kG$-module. If $M{\downarrow}_{H}^G$ is injective for each $H \leq G$ such that $H \in \mathfrak{G}$, then $M$ is injective.
\end{lemma}
\begin{proof}
    We prove this by transfinite induction; it is true by assumption if $G$ is an $\textup{H}_0\mathfrak{G}$ group. Suppose that it is true for all $\textup{H}_{\beta}\mathfrak{G}$ subgroups for all $\beta < \alpha$, for some ordinal $\alpha$. Let $G$ be an $\textup{H}_{\alpha}\mathfrak{G}$ group. 
    \par 
    Since $M$ is Gorenstein injective, it is injective over $k$ and hence $\textup{Hom}_k(-,M)$ is exact. Therefore, we have an exact sequence of the form $0 \to M \to C_0 \to \dots C_n \to 0$ such that each $C_i$ is a product of modules of the form $M{\downarrow}_H^G{\Uparrow}_H^G$ where $H \leq G$ is an $\textup{H}_{\beta}\mathfrak{G}$ subgroup for some $\beta < \alpha$. By induction, each such $M{\downarrow}_H^G{\Uparrow}_H^G$ is injective and therefore $M$ has finite injective dimension. It follows that $M$ is injective, since each Gorenstein injective of finite injective dimension must be injective. 
\end{proof}
\begin{proposition}\label{corfsubinj}
    Suppose $G$ is an $\lhg$ group. Then $M \in {}^{\perp}\mathcal{GI}_G$ if and only if $M{\downarrow}_H^G \in {}^{\perp}\mathcal{GI}_H$ for every $H \leq G$ with $H \in \mathfrak{G}$.
\end{proposition}
\begin{proof}
Let $M \in {}^{\perp}\mathcal{GI}_G$ and take any $H \leq G$ such that $H \in \mathfrak{G}$. Let $N$ be any Gorenstein injective $kH$-module. Then $\textup{Ext}^1_{kH}(M{\downarrow}_H^G,N) \cong \textup{Ext}^1_{kG}(M,N{\Uparrow}_H^G) = 0$, since coinduction preserves Gorenstein injectivity. This shows one implication. 
\par 
For the reverse implication, we first suppose that $G$ is an $\textup{H}\mathfrak{G}$ group. 
    Suppose $M{\downarrow}_H^G \in {}^{\perp}\mathcal{GI}_H$ for every $H \leq G$ with $H \in \mathfrak{G}$. Consider the short exact sequence $0 \to M \to X \to Y \to 0$ with $X \in \mathcal{GI}_G$ and $Y \in {}^{\perp}\mathcal{GI}_G$. We have shown that $Y{\downarrow}_H^G \in {}^{\perp}\mathcal{GI}$. Since ${}^{\perp}\mathcal{GI}$ is thick, our assumption on $M$ then implies that $X{\downarrow}_H^G \in {}^{\perp}\mathcal{GI}_H$ for every $H \leq G$ with $H \in \mathfrak{X}$. 
    \par 
    From \Cref{dugi} we know that $X{\downarrow}_H^G$ is Gorenstein injective. It is shown in the proof of \cite[Theorem~4.6]{stovsarmodel} that $\mathcal{GI}_H \cap {}^{\perp}\mathcal{GI}_H$ coincides with the class of injectives. Therefore, $X{\downarrow}_H^G$ is injective. 
    \par 
    Applying \Cref{dugi2} shows that $X$ is itself injective. It is shown in \cite[Lemma~4.4]{stovsarmodel} that ${}^{\perp}\mathcal{GI}_G$ is thick and hence $M \in {}^{\perp}\mathcal{GI}_G$ as desired. 
    \par 
    Now, suppose $G$ is an $\lhg$ group and write it as a direct limit of its finitely generated subgroups, i.e. $G = \varinjlim G_{\alpha}$. Then we can write $M$ as a direct limit $\varinjlim M{\downarrow}_{G_{\alpha}}^G{\uparrow}_{G_{\alpha}}^G$. Note that each $G_{\alpha}$ is necessarily an $\textup{H}\mathfrak{G}$ group and we have just shown that $M{\downarrow}_{G_{\alpha}}^G \in {}^{\perp}\mathcal{GI}_{G_{\alpha}}$.
    \par We claim that $M{\downarrow}_{G_{\alpha}}^G{\uparrow}_{G_{\alpha}}^G \in {}^{\perp}\mathcal{GI}_{G}$. Let $N$ be a Gorenstein injective $kG$-module and so $N{\downarrow}_{G_{\alpha}}^G$ is Gorenstein injective by \Cref{dugi}. Therefore, $\textup{Ext}^1_{kG}(M{\downarrow}_{G_{\alpha}}^G{\uparrow}_{G_{\alpha}}^G,N) \cong \textup{Ext}^1_{kG_{\alpha}}(M{\downarrow}_{G_{\alpha}}^G,N{\downarrow}_{G_{\alpha}}^G) = 0$ as desired.
    \par 
    Finally, by \cite[Lemma~4.1]{stovsarmodel} we know that ${}^{\perp}\mathcal{GI}$ is closed under direct limits and so $M \in {}^{\perp}\mathcal{GI}$.  
\end{proof}
\begin{corollary}\label{eqperpsd}
    We have $\mathcal{GP}^{\perp} = {}^{\perp}\mathcal{GI}$ for every $\lhg$ group $G$. Furthermore, there is a triangulated equivalence $\gp(kG) \cong \gi(kG)$ for $\lhg$ groups.

\end{corollary}
\begin{proof}
    The first statement follows from \Cref{corfsub} and \Cref{corfsubinj}, since these show that $M \in \mathcal{GP}_{G}^{\perp}$ if and only if $M{\downarrow}_H^G \in \mathcal{GP}_H^{\perp}$ for every $\mathfrak{G}$-subgroup $H \leq G$ if and only if $M{\downarrow}_H^G \in {}^{\perp}\mathcal{GI}_H$ for every $\mathfrak{G}$ subgroup $H \leq G$ if and only if $M \in {}^{\perp}\mathcal{GI}_G$. 
    \par 
 This means that the trivial modules in the two model structures are the same. It is not too difficult to see that the identity functor must therefore be a Quillen equivalence, see \cite[Lemma~5.4]{Estrada_Gillespie_2019} for a proof. 
    \par
    This shows that there is an equivalence between the homotopy categories of the two model structures. Since the homotopy category of the Gorenstein projective model structure is equivalent to $\gp(kG)$ via taking Gorenstein projective special precovers (and similarly for the injective case) we can see that the equivalence in the statement is given by taking such approximations.
    \par 
    The fact that it is a triangulated equivalences follows since the identity functor is exact and sends trivial objects to trivial objects. By the construction of the exact triangles in the homotopy category it is then clear that the identity functor sends standard exact triangles to standard exact triangles.
\end{proof}
An Artin algebra over which $\mathcal{GP}^{\perp} = {}^{\perp}\mathcal{GI}$ is called virtually Gorenstein. As in \cite[Example~3.10]{stovsarmodel}, an Artin algebra is virtually Gorenstein if and only if $(\mathcal{GP},\mathcal{GP}^{\perp})$ is of finite type. In view of \Cref{eqperpsd}, it would be interesting to know if there were a similar characterisation in general; we note that we have the following.
\begin{proposition}\label{gpandpgf}
    Suppose $\mathcal{GP}^{\perp} = {}^{\perp}\mathcal{GI}$. Then $(\mathcal{GP},\mathcal{GP}^{\perp})$ is of countable type and $\mathcal{GP} = \mathcal{PGF}$.
\end{proposition}
\begin{proof}
    First, we recall that every projectively coresolved Gorenstein flat module is Gorenstein projective by \cite[Theorem~3.4]{stovsarmodel}. 
    \par By \cite[Lemma~4.1]{stovsarmodel} we know that ${}^{\perp}\mathcal{GI}$ is closed under direct limits and therefore so must be $\mathcal{GP}^{\perp}$. Now, it is shown in \cite[Theorem~6.1]{aroch2016ApproximationsAM} that this implies that $(\mathcal{GP},\mathcal{GP}^{\perp})$ is of countable type and $\mathcal{GP}^{\perp}$ is definable. Hence, every module in the definable closure of $kG$ is in $\mathcal{GP}^{\perp}$. It follows from the characterisation of projectively coresolved Gorenstein flat modules in \cite[Corollary~3.5]{stovsarmodel} that Gorenstein projectives are projectively coresolved Gorenstein flat. 
\end{proof}
\begin{corollary}\label{th2d}
    Let $G$ be an $\textup{LH}\mathfrak{G}$ group, where $k$ is a commutative ring of finite global dimension. Then the class of Gorenstein projective modules and the class of projectively coresolved Gorenstein flat modules coincide. Hence, every Gorenstein projective module is Gorenstein flat.
\end{corollary}
\begin{proof}
    This follows from \Cref{eqperpsd} and \Cref{gpandpgf}. Furthermore, by definition every projectively coresolved Gorenstein flat module is Gorenstein flat.
\end{proof}
In particular, the homotopy categories of the two model structures in \cite[Section~3]{stovsarmodel} - explicitly these are $\mathcal{PGF}$ modulo the projectives and $\mathcal{GF}$-cotorsion modulo the flat-cotorsion modules - are equivalent to $\gp(kG)$.

\subsection{Gorenstein AC-model structures}\label{ACsec}
For this section, we consider $\textup{LH}\mathfrak{G}_{AC}$ groups. This is because we have shown in \Cref{gpequalgacp} that $\mathcal{GP} = \mathcal{GP}_{AC}$ for $\mathfrak{G}_{AC}$ groups, whereas we don't know if this is true for groups of finite Gorenstein cohomological dimension. 
\par 
In \cite{Bravo2014TheSM} the authors construct an assortment of model structures in order to generalise the stable module category of a quasi-Frobenius ring to arbitrary rings. 
\par 
For the following, recall the definitions of Gorenstein AC-projective and Gorenstein AC-injective modules from \Cref{gorensteindef}.
\begin{theorem}\cite[Theorem~5.5, Theorem~8.5]{Bravo2014TheSM} 
    There is a model structure on $\textup{Mod}(kG)$ such that every module is fibrant, the cofibrant modules are Gorenstein AC-projectives and the trivial modules are those contained in $\mathcal{GP}_{AC}^{\perp}$.
    \par 
    Similarly, there is a model structure such that every module is cofibrant, the fibrant objects are the Gorenstein AC-injectives and the trivial modules are those contained in ${}^{\perp}\mathcal{GI}_{AC}$.
\end{theorem}
We refer to the model structures of the above theorem as the Gorenstein AC-projective model structure and Gorenstein AC-injective model structure respectively.
\par
For the following, we need to recall that for any (left) module $M$ we define its character module as $M^+ = \textup{Hom}_{\mathbb{Z}}(M,\mathbb{Q}/\mathbb{Z})$, with the obvious (right) module structure.
\begin{lemma}\label{cutelem}
    Suppose $A$ is an absolutely clean $kG$-module. Then $A{\downarrow}_F^G$ is absolutely clean for any subgroup $F \leq G$. Similarly, if $L$ is a level $kG$-module then $L{\downarrow}_F^G$ is level for any subgroup $F \leq G$.
\end{lemma}
\begin{proof}
    By definition, $A$ is a module such that $\textup{Ext}^1_{kG}(M,A) = 0$ for every $kG$-module $M$ of type $\textup{FP}_{\infty}$. Let $F \leq G$ be a subgroup and $N$ a $kF$-module of type $\textup{FP}_{\infty}$. It follows that $N{\uparrow}_F^G$ is $\textup{FP}_{\infty}$ over $kG$ and hence $\textup{Ext}^1_{kF}(N,A{\downarrow}_F^G) \cong \textup{Ext}^1_{kG}(N{\uparrow}_F^G,A) = 0$, which exactly means that $A{\downarrow}_F^G$ is absolutely clean.
    \par
    Now let $L$ be a level $kG$-module. It is shown in \cite[Theorem~2.12]{Bravo2014TheSM} that $L$ is level if and only if $L^+$ is absolutely clean. Therefore, $L^+{\downarrow}_F^G$ is absolutely clean. It is easy to see that $L^+{\downarrow}_F^G \cong (L{\downarrow}_F^G)^+$ and so again by \cite[Theorem~2.12]{Bravo2014TheSM} we find that $L{\downarrow}_F^G$ is level.
\end{proof}
We also require the following lemma; it is unknown in general if the global Gorenstein projective dimension can differ from the global Gorenstein AC-projective dimension. The following shows that they must be equal, as long as we restrict ourselves to $\textup{LH}\mathfrak{G}_{AC}$ groups. 
\begin{lemma}\label{ddff}
    Suppose $G$ is an $\textup{LH}\mathfrak{G}_{AC}$ group with $\textup{Gcd}_k(G) < \infty$. Then $\textup{glGPD}_{AC}(kG) < \infty$.
\end{lemma}
\begin{proof}
    We have seen in \Cref{fintypec} that for such a group the cotorsion pair $(\mathcal{GP},\mathcal{GP}^{\perp})$ is cogenerated by a set of $\textup{FP}_{\infty}$ modules. This implies that any absolutely clean module is contained in $\mathcal{GP}^{\perp}$. We see from \Cref{fpdext} and \Cref{perpequal} that this means that any absolutely clean module has finite injective dimension. Note that these injective dimensions are uniformly bounded, i.e. there exists some $n > 0$ such that each absolutely clean module has injective dimension $\leq n$.
    \par 
    Let $L$ be a level $kG$-module and consider the following exact sequence, where the integer $n$ is as above, and $P_*$ is a projective resolution of $L$.
    \[0 \to \Omega^n(L) \to P_n \to \dots \to P_0 \to L \to 0\]
    Taking the character modules of everything gives the following exact sequence. 
    \[0 \to L^+ \to P_0^+ \to \dots P_n^+ \to \Omega^n(L)^+ \to 0\]
    We know from \cite[Theorem~2.12]{Bravo2014TheSM} that $L^+$ is absolutely clean. Furthermore, it is well known a module is flat if and only if its character module is injective. Since $L^+$ has finite injective dimension we know that $\Omega^n(L)^+$ is injective, and hence $\Omega^n(L)$ is flat. This means that $L$ has finite flat dimension. However, since $\textup{glGPD}(kG) < \infty$, we know from \cite[Corollary~4.3]{em22} that $L$ must have finite projective dimension. In particular, $L \in \mathcal{GP}^{\perp}$ and so every Gorenstein projective is Gorenstein AC-projective. This completes the proof. 
\end{proof}
\begin{proposition}\label{gpeqgpac} Let $k$ be a commutative ring of finite global dimension. Then $\mathcal{GP} = \mathcal{GP}_{AC}$ for each $\textup{LH}\mathfrak{G}_{AC}$ group $G$.
\end{proposition}
\begin{proof}
    We have shown in \Cref{gpequalgacp} that $\mathcal{GP} = \mathcal{GP}_{AC}$ for each $\mathfrak{G}_{AC}$ group. 
    \par 
    Now, we let $G$ be a $\textup{LH}\mathfrak{G}_{AC}$ group. 
    By definition, we have that $\mathcal{GP}_{AC} \subseteq \mathcal{GP}$. We claim that $\mathcal{GP}_{AC}^{\perp} \subseteq \mathcal{GP}^{\perp}$. This implies that $\mathcal{GP} \subseteq \mathcal{GP}_{AC}$ by using the two cotorsion pairs $(\mathcal{GP},\mathcal{GP}^{\perp})$ and $(\mathcal{GP}_{AC},\mathcal{GP}_{AC}^{\perp})$. 
    \par 
    Suppose $M \in \mathcal{GP}_{AC}^{\perp}$. We argue that $M{\downarrow}_H^G \in \mathcal{GP}^{\perp}$ for each $\mathfrak{G}$ subgroup $H \leq G$. It then follows from \Cref{corfsub} that $M \in \mathcal{GP}^{\perp}$ as desired. There is a subtlety here, since we don't know that the statement of the proposition holds for $\mathfrak{G}$ groups in general. However, we use \Cref{ddff} to see that for each $\mathfrak{G}$ subgroup $H \leq G$, we in fact have that $\textup{glGPD}_{AC}(kH) < \infty$. Hence, there is an equality $\mathcal{GP} = \mathcal{GP}_{AC}$ over $H$. 
    \par
    We see that it therefore suffices to show that $M{\downarrow}_H^G \in \mathcal{GP}_{AC}^{\perp}$. 
    \par 
    \sloppy Let $X$ be a Gorenstein AC-projective $kH$-module. We know that $\textup{Ext}^1_{kH}(X,M{\downarrow}_H^G) \cong \textup{Ext}^1_{kG}(X{\uparrow}_H^G,M)$; to show that this vanishes it is sufficient to show that $X{\uparrow}_H^G$ is Gorenstein AC-projective over $kG$, since we have assumed $M \in \mathcal{GP}_{AC}^{\perp}$. 
    \par 
    Suppose that $X$ is the kernel in an acyclic complex of $kH$-projectives $P_*$ which is $\textup{Hom}_{kH}(-,L)$ acyclic for each level $kH$-module $L$. Then $X{\uparrow}_H^G$ is the kernel in an exact complex of projectives $P_*{\uparrow}_H^G$. For any level $kG$-module $L'$, we have that $\textup{Hom}_{kG}(P_*{\uparrow}_H^G,L') \cong \textup{Hom}_{kH}(P_*,L'{\downarrow}_H^G)$. From \Cref{cutelem} we know that $L'{\downarrow}_H^G$ is level and hence this is exact, i.e. $X{\uparrow}_H^G$ is Gorenstein AC-projective as desired. 
\end{proof}
We note that since every Gorenstein AC-projective is projectively coresolved Gorenstein flat by definition, this provides another proof that $\mathcal{GP} = \mathcal{PGF}$ for $\textup{LH}\mathfrak{G}_{AC}$ groups. 
\par 
A dual argument now gives the following.
\begin{proposition}\label{giegiac}
Suppose $k$ is a commutative ring of finite global dimension. We have that $\mathcal{GI} = \mathcal{GI}_{AC}$ for each $\textup{LH}\mathfrak{G}_{AC}$ group $G$.
\end{proposition}

\section{Monoidal model category}\label{monoidalsection}
One of the main motivations for this work was in constructing a suitable tensor-triangulated stable category for $\lhf$ groups. We have already shown that the Gorenstein projective model structure is monoidal for groups of finite Gorenstein cohomological dimension, and hence the homotopy category will be tensor-triangulated. We now extend this to $\lhg$ groups. 
\begin{lemma}\label{l2}
    If $G$ is an $\lhg$ group and $M,N$ are Gorenstein projective $kG$-modules, then $M \otimes N$ is Gorenstein projective.
\end{lemma}
\begin{proof}
    First, let $H$ be an $\mathfrak{G}$ group and let $X$ and $Y$ be two Gorenstein projectives, i.e. cofibrants in the model structure. The pushout product of the two maps $0 \to X$ and $0 \to Y$, which are cofibrations, shows that $X \otimes Y$ is also cofibrant (and hence Gorenstein projective). 
    \par
    We prove this first for $\textup{H}\mathfrak{G}$ groups, and proceed by transfinite induction; the base case is true by \Cref{basemono}. Suppose that $G$ is an $\textup{H}_{\alpha}\mathfrak{G}$ group and that the statement holds for all $\textup{H}_{\beta}\mathfrak{G}$ groups for $\beta < \alpha$. 
    \par 
    We claim that $M \otimes N$ is a transfinite extension of Gorenstein projectives. We recall from \Cref{induc}, and the remark after, that $(\mathcal{GP}_G,\mathcal{GP}^{\perp}_G)$ is cogenerated by the set $\bigcup\limits_{H \in \mathfrak{G}}\mathcal{X}_H{\uparrow}_H^G$. We may assume that this contains $kG$ (which is a generator of $\textup{Mod}(kG))$ and so it is shown in \cite[Corollary~2.15]{Saorn2010OnEC} that every object in $\mathcal{GP}_G$ is a retract of a transfinite extension of modules in $\bigcup\limits_{H \in \mathfrak{G}}\mathcal{X}_H{\uparrow}_H^G$.
    \par 
    In particular, $M$ is a retract of $M'$, which is such a transfinite extension. But then $M \otimes N$ is a retract of $M' \otimes N$. We know that $N$ is projective over $k$, since $k$ has finite global dimension, and so tensoring the transfinite extension of $M'$ with $N$ shows that $M' \otimes N$ has a transfinite extension by modules of the form $X{\uparrow}_H^G \otimes N$, where $X \in \mathcal{X}_H$ for some $H$. 
    \par 
    There is an isomorphism $X{\uparrow}_H^G \otimes N \cong (X \otimes N{\downarrow}_H^G){\uparrow}_H^G$. From \Cref{resgp1} we know that $N{\downarrow}_H^G$ is Gorenstein projective and so by the inductive hypothesis it follows that $X \otimes N{\downarrow}_H^G$ is Gorenstein projective. Hence, so is $(X \otimes N{\downarrow}_H^G){\uparrow}_H^G$. The tensor product commutes with direct limits and so $M' \otimes N$ is a transfinite extension of Gorenstein projectives as desired. It follows that $M' \otimes N$ is actually Gorenstein projective by \cite[Corollary~2.5]{Eprecov}. Finally, $M \otimes N$ is a summand of a Gorenstein projective and hence is itself Gorenstein projective.
    \par 
    Now let $G$ be an $\lhg$ group and we induct on the cardinality; it is true for countable $\lhg$ groups since these lie in $\textup{H}\mathfrak{G}$. Suppose that $G$ is uncountable and write $G = \bigcup\limits_{\alpha<\gamma}G_{\alpha}$ as a union of groups of strictly smaller cardinality. From \Cref{resgp1} we know that $M{\downarrow}_{G_{\alpha}}^G$ and $N{\downarrow}_{G_{\alpha}}^G$ are both Gorenstein projective. Hence, by the inductive hypothesis we have that $(M \otimes N){\downarrow}_{G_{\alpha}}^G$ is Gorenstein projective. 
    \par 
    It is shown in \cite[Proposition~3.1]{ETgcd} that $M \otimes N$ then has Gorenstein projective dimension $\leq 1$. The same argument applies to $\Omega^{-1}(M) \otimes N$. The short exact sequence $0 \to M \otimes N \to P \otimes N \to \Omega^{-1}(M) \otimes N \to 0$ implies now that $M \otimes N$ is Gorenstein projective.
\end{proof}
\begin{remark}
    For $\lhf$ groups, there is a different argument to show that Gorenstein projectives are closed under tensoring. Indeed, we know from \cite[Corollary~D]{demtalres} that for such groups, any acyclic complex of projectives is automatically totally acyclic. Then if $M$ and $N$ are Gorenstein projective, with $M$ a kernel in an acyclic complex of projectives $P_*$, we see that $P_* \otimes N$ is an acyclic (and hence totally acyclic) complex of projectives. This shows that $M \otimes N$ is Gorenstein projective. It is conjectured in \cite{demtalres} that acyclic complexes of projectives are totally acyclic over any group.
\end{remark}
\begin{lemma}\label{closymmon}
   Let $G$ be an $\lhg$ group. Then the following statements hold.
    \begin{enumerate}[label=(\roman*)]
        \item Every Gorenstein projective module is flat over $k$ and for every $kG$-module $M$ there is a surjective map $X \to M$ with $X$ Gorenstein projective
        \item If $X,Y$ are Gorenstein projective, so is $X \otimes Y$
        \item If $X,Y$ are Gorenstein projective, and at least one of them is projective then $X \otimes Y$ is projective 
        \item For any Gorenstein projective $X$ the natural map $\mathcal{G}(k)\otimes X \to k \otimes X$ is a weak equivalence, where $\mathcal{G}(k) \to k$ is the Gorenstein projective special precover
    \end{enumerate}
\end{lemma}
\begin{proof}
\begin{enumerate}[label=(\roman*)]
    \item Since $k$ has finite global dimension, every Gorenstein projective is projective (and hence flat) over $k$. The existence of Gorenstein projective precovers shows the existence of the required surjective maps.
    \item This is \Cref{l2}.
    \item This is shown in \cite[Lemma~4.2]{MSstabcat}.
    \item Let $K$ be the kernel of the map $\mathcal{G}(k) \to k$. We claim that $K \otimes X \in \mathcal{GP}^{\perp}$; this shows that the map $\mathcal{G}(k) \otimes X \to k \otimes X$ is a trivial fibration and in particular is a weak equivalence as desired. 
    \par 
    In order to show that $K \otimes X \in \mathcal{GP}^{\perp}$ we show that $(K \otimes X){\downarrow}_H^G \in \mathcal{GP}^{\perp}_H$ for all $H \in \mathfrak{G}$ and then apply \Cref{corfsub}. Now, we can see that $\mathcal{G}(k){\downarrow}_H^G \to k{\downarrow}_H^G$ acts as a cofibrant replacement for $k{\downarrow}_H^G \cong k$. Therefore, by \Cref{basemono} we have that $\mathcal{G}(k){\downarrow}_H^G \otimes X{\downarrow}_H^G \to k \otimes X{\downarrow}_H^G$ is a weak equivalence. Since it is also a fibration, it follows that it is a trivial fibration and so its kernel is in $\mathcal{GP}^{\perp}$. However, the kernel is simply $(K\otimes X){\downarrow}_H^G$.
\end{enumerate}
\end{proof}
\begin{theorem}\label{ttcatth}
    The Gorenstein projective model structure is a monoidal model category for every $\lhg$ group $G$.
\end{theorem}
\begin{proof}
    We have shown the unit axiom in \Cref{closymmon}$(iv)$. To show the pushout-product axiom we wish to apply \cite[Theorem~7.2]{hovmod}, in the same manner as we did in \Cref{basemono}. In fact, in \Cref{closymmon} we have shown that all the required conditions hold.
\end{proof}
In particular, the stable category $\gp(kG)$ is a tensor-triangulated category, where the tensor unit is given by the Gorenstein projective special precover of $k$. 
\section{Compact objects}\label{compasecti}
We have now constructed a well-behaved stable category for a large class of infinite groups, namely for $\textup{LH}\mathfrak{G}$ groups where $\mathfrak{G}$ is the class of groups of finite Gorenstein cohomological dimension. In doing so, we have seen that much of the structure of this category is controlled by the $\mathfrak{G}$-subgroups. 
\par 
We now explore the stable category further and highlight some key differences between the stable category for infinite groups and that of finite groups.
\par
We start off by determining the compact objects; first, we recall the notion of compactness in triangulated categories.
\begin{definition}
    Let $\mathcal{T}$ be a triangulated category with infinite coproducts. An object $M \in \mathcal{T}$ is called compact if for any $\bigoplus N_i$ the natural map $\bigoplus \textup{Hom}_{\mathcal{T}}(M,N_i) \to  \textup{Hom}_{\mathcal{T}}(M,\bigoplus N_i)$ is an isomorphism. 
\end{definition}
\begin{definition}\label{tldef}
A full, triangulated subcategory $\mathcal{S}$ is called thick if it is closed under taking direct summands. If $\mathcal{T}$ has coproducts, then $\mathcal{S}$ is called localising if it is thick and it is closed under taking coproducts. 
\end{definition}
The following is well known and may be taken as a definition for what it means for a triangulated category to be compactly generated. 
\begin{lemma}(e.g. \cite[Corollary~3.4.8]{krause_2021})\label{lem:SSlocgen}
Suppose $\mathcal{C}$ is a set of compact objects in a triangulated category $\mathcal{T}$, with suspension $\Sigma$. Then the following are equivalent:
\begin{enumerate}
\item The smallest localising subcategory containing $\mathcal{C}$ is $\mathcal{T}$.
\item For any object $X$ in $\mathcal{T}$ such that $\textup{Hom}_{\mathcal{T}}(\Sigma^nC,X) = 0$ for all $n\in\mathbb{Z}$ and all $C \in \mathcal{C}$ we have $X = 0$.
\end{enumerate}
\end{lemma}
To be able to consider the compact objects in the stable category, we need to first know that coproducts exist. 
\begin{lemma}
    Let $G$ be an $\lhg$ group. Then coproducts and products exist in the stable category.
\end{lemma}
\begin{proof}
    It is straightforward to check that $\gp(kG)$ has coproducts, see for example \cite[Proposition~4.5]{MSstabcat}, the proof of which works in this situation without change. 
    \par 
    In fact, consider a collection of modules $X_i$ for $i \in I$, with Gorenstein projective special precovers given by $\mathcal{G}(X_i)$. We know from \Cref{eqperpsd} that $\mathcal{GP}^{\perp} = {}^{\perp}\mathcal{GI}$, from which we infer that $\bigoplus \mathcal{G}(X_i)$ is a Gorenstein projective special precover of $\bigoplus X_i$. Hence, these are isomorphic in the stable category, and so we may calculate the coproducts as usual in the module category. 
    \par 
    The statement about products follows similarly, using that the stable category is equivalent to $\gi(kG)$. 
\end{proof}
We will characterise the compact objects in the homotopy category in terms of certain cohomology functors commuting with filtered colimits of modules. In order to do this, we use the fact that given an hereditary abelian model structure, it is possible to define a generalisation of Tate cohomology using the completeness of the cotorsion pairs. This is explained in detail in \cite{Gillespie2020CanonicalRI}; we briefly recall the main idea in the context of the Gorenstein projective model structure. 
\par 
\sloppy Suppose $M$ is Gorenstein projective. Using the completeness of the cotorsion pair $(\textup{Proj}(kG),\textup{Mod}(kG))$ we form a projective resolution $P_* \to M$. Since Gorenstein projectives are closed under kernels of epimorphisms, e.g. \cite[Theorem~2.5]{HOLM2004167}, we see that each kernel in $P_*$ is also Gorenstein projective. 
\par 
Using the completeness of the cotorsion pair $(\mathcal{GP},\mathcal{GP}^{\perp})$, we can construct a coresolution $M \to Q_*$ such that $Q_i \in \mathcal{GP}^{\perp}$ for each $i$ and each kernel is Gorenstein projective. In fact, since the Gorenstein projectives are closed under extensions, each $Q_i$ is also Gorenstein projective and is hence projective. Therefore, splicing together $P_*$ and $Q_*$ gives us a totally acyclic complex of projectives, with $M$ being the kernel in degree zero. We call this $W_M$.
\par 
For any other $kG$-modules $N$ and $X$, we can then define 
\[\textup{Ext}^n_{\mathcal{GP}}(N,X) = H^n(\textup{Hom}_{kG}(W_{\mathcal{G}(N)},X))\]
where $\mathcal{G}(N)$ denotes the Gorenstein projective special precover of $N$, see \cite[Definition~6.4, Theorem~6.5(3)]{Gillespie2020CanonicalRI}.
\par 
We note that by \cite[Corollary~6.6]{Gillespie2020CanonicalRI}, we have that $\textup{Ext}^0_{\mathcal{GP}}(N,X) = \widetilde{\textup{Hom}}_{kG}(N,X)$. 
\par 
For a group with $\textup{Gcd}_k(G) < \infty$ this has a different description: it is simply the complete cohomology, defined as follows.
\begin{definition}\label{defcompco}
    Let $M$ be a $kG$-module of finite Gorenstein projective dimension. A complete resolution is a totally acyclic complex of projectives $P_*$ which agrees in high dimensions with a projective resolution of $M$. We define the complete cohomology 
    as 
    \[\widehat{\textup{Ext}}_{kG}^n(M,-) := H^n(\textup{Hom}_{kG}(P_*,-))\]
\end{definition}
Note that this definition is only suitable when complete resolutions exist, and it agrees with the definition of complete cohomology we gave in \Cref{firstcomcodef} in this case, as is shown in \cite[Theorem~1.2]{CKcompres}. \par
Since every module has finite Gorenstein projective dimension over a group of finite Gorenstein cohomological dimension, see \Cref{ETThe}, it is easy to see that complete cohomology can be defined for any module over such groups. 
\begin{lemma}\label{thetwoext}
    Suppose $M$ is a module of finite Gorenstein projective dimension. Then 
    \[\textup{Ext}^n_{\mathcal{GP}}(M,-) \cong \widehat{\textup{Ext}}_{kG}^n(M,-)\]
\end{lemma}
\begin{proof}
    Suppose $M$ has projective resolution $P_*$, and so $\Omega^d(M)$ is Gorenstein projective, where $d = \textup{Gpd}_{kG}(M)$. As above, we can build a projective coresolution for $\Omega^d(M)$ which is totally acyclic and we splice this together with the projective resolution of $M$ in degrees $\geq d$. We call this complex $Q_*$. 
\par 
Using total acyclicity of $Q_*$ we can lift the map $Q_{\geq d} \to P_{\geq d}$ to a map $Q_* \to P_*$, which we may assume to be surjective. Taking the kernel in degree zero then gives a surjective map $X \to M$ where $X$ is Gorenstein projective. The kernel of this map has finite projective dimension and hence this is a special Gorenstein projective precover of $M$; that is, we have constructed an explicit cofibrant replacement in the model structure. We then see that $\textup{Ext}^n_{\mathcal{GP}}(M,X)$ is, by definition, the cohomology of the cochain complex $\textup{Hom}_{kG}(Q_*,X)$, i.e. the complete cohomology.
\end{proof}
In particular, if $M$ is Gorenstein projective then we have the following isomorphisms. 
\begin{equation}\label{gpla1}
    \textup{Ext}^n_{\mathcal{GP}}(M,-) \cong \textup{Ext}^n_{kG}(M,-) \textup{ for } n \geq 1 
\end{equation}
and 
\begin{equation}\label{gpla2}
    \textup{Ext}^0_{\mathcal{GP}}(M,-) \cong \underline{\textup{Hom}}_{kG}(M,-)
\end{equation}
\begin{definition}
    Let $M$ be a $kG$-module. We say $M$ is $\mathcal{GP}$-finitary if, for any filtered colimit of modules $\varinjlim N_i$, the natural map 
    \[\varinjlim \textup{Ext}^n_{\mathcal{GP}}(M,N_i) \to \textup{Ext}^n_{\mathcal{GP}}(M,\varinjlim N_i)\]
    is an isomorphism.
    \par 
    Following \cite{Hamfin}, we say that $M$ is completely finitary if the natural map 
    \[\varinjlim \widehat{\textup{Ext}}^n_{kG}(M,N_i) \to \widehat{\textup{Ext}}^n_{kG}(M,\varinjlim N_i)\]
    is an isomorphism.
\end{definition}
Note that we mean filtered colimits of modules in the usual module category. If $M$ has finite Gorenstein projective dimension, then being $\mathcal{GP}$-finitary is the same as being completely finitary.
\par
The above definition is a generalisation of the notion of being $\textup{FP}_{\infty}$, as illustrated by the following result of Brown. 
\begin{proposition}\label{extcrit}\cite[Corollary]{Brown1975}
The following are equivalent for a $kG$-module $M$. 
\begin{enumerate}
\item $M$ is of type $\textup{FP}_{\infty}$
\item The functors $\textup{Ext}^n_{kG}(M,-)$ commute with directed colimits for all $n \geq 0$
\end{enumerate}
\end{proposition}
By \cite[Corollary~1.5]{ARcats}, a functor commutes with directed colimits if and only if it commutes with filtered colimits, and so we will tend to refer to filtered colimits from now on. 
\par 
We have the following important examples of $\mathcal{GP}$-finitary modules. 
\begin{lemma}\label{fhsao}
    Suppose $M$ is an $\textup{FP}_{\infty}$ Gorenstein projective module. Then $M$ is $\mathcal{GP}$-finitary.
\end{lemma}
\begin{proof}
    Using \eqref{gpla1} along with \Cref{extcrit} shows that $\textup{Ext}^n_{\mathcal{GP}}(M,-)$ commutes with filtered colimits for all $n \geq 1$. It is shown in \cite[Lemma~1.2]{em22} that for any $\textup{FP}_{\infty}$ module $M$, $\underline{\textup{Hom}}_{kG}(M,-)$ commutes with filtered colimits; from \eqref{gpla2} we see that $\textup{Ext}^0_{\mathcal{GP}}(M,-) $ then also commutes with filtered colimits. 
    \par 
    We are left to show that $\textup{Ext}^n_{\mathcal{GP}}(M,-)$ commutes with filtered colimits for $n < 0$. However, since $M$ is $\textup{FP}_{\infty}$ and Gorenstein projective we may embed $M$ in a finitely generated projective with Gorenstein projective cokernel cf. \cite[Theorem~4.2(iv)]{ADTgor}. That is, we have a short exact sequence $0 \to M \to P \to M' \to 0$ with $M'$ also an $\textup{FP}_{\infty}$ Gorenstein projective. There is a long exact sequence in cohomology for our generalised Tate cohomology functors as shown in \cite[Theorem~6.5]{Gillespie2020CanonicalRI}. Furthermore, these functors vanish on trivial objects. It follows that $\textup{Ext}^{-1}_{\mathcal{GP}}(M,-) \cong \textup{Ext}^0_{\mathcal{GP}}(M',-)$, which we have shown commutes with filtered colimits. A similar argument shows that in fact $\textup{Ext}^n_{\mathcal{GP}}(M,-)$ commutes with filtered colimits for all $n < 0$. 
\end{proof}
In order to identify the compact objects we will use a deep theorem of Saor\'in, Št'ovíček, and Virili \cite[Corollary~6.10]{SAORINSTOV}. For $\textup{LH}\mathfrak{G}_{AC}$ groups we could instead proceed by showing that the stable category is compactly generated by $\textup{FP}_{\infty}$ modules and then using that $\mathcal{GP}$-finitary modules form a thick subcategory. However, we do not know if the stable category will be compactly generated for $\textup{LH}\mathfrak{G}$ groups in general. 
\par 
In \cite{SAORINSTOV}, the authors work in the context of derivators. We refer to \cite[Section~2]{SAORINSTOV} for the necessary background, but we will briefly recall what is necessary for our purposes. Let $\textup{Cat}^{op}$ be the 2-category of small categories with the direction of functors reversed (but not the direct of natural transformations). Ignoring set theoretic issues, we also let $\textup{CAT}$ be the 2-category of all categories. A derivator is a strict 2-functor $\mathbb{D}: \textup{Cat}^{op} \to \textup{CAT}$ satisfying some further axioms. 
\par 
For what we are interested in, the most important example of a derivator is as follows, see \cite[Example~2.15]{SAORINSTOV}. Suppose we have a model category structure on a category $\mathcal{M}$ with weak equivalences $\mathcal{W}$. For any small category $I$ we let $\mathcal{M}^I$ be the category of functors $I \to \mathcal{M}$, and let $\mathcal{W}_I$ be the class of morphisms in $\mathcal{M}^I$ which are pointwise in $\mathcal{W}$. Cisinski \cite{Cisinski2003} shows that we can construct the universal localisation $\mathcal{M}^I[\mathcal{W}_I^{-1}]$ and that $I \mapsto \mathcal{M}^I[\mathcal{W}_I^{-1}]$ in fact defines a derivator. This is always a strong derivator, and will be stable if the model category is stable. 
\par 
In particular, if we let $\mathbbm{1}$ be the unique category with one element and no non-identity morphisms, then $\mathcal{M}^{\mathbbm{1}}[\mathcal{W}_{\mathbbm{1}}^{-1}]$ will be the homotopy category of the model category. 
\par 
Note that if $\mathcal{M}^I$ has, for example, the projective model structure where weak equivalences and fibrations are determined pointwise, then $\mathcal{M}^I[\mathcal{W}_I^{-1}] = \textup{Ho}(\mathcal{M}^I)$.
\par 
We also recall the definition of homotopy colimits. Given a small category $I$, we write $\textup{pt}_I: I \to \mathbbm{1}$ to be the unique functor. The axioms of a derivator ensure that the induced functor $\mathbb{D}(\textup{pt}_I): \mathbb{D}(\mathbbm{1}) \to \mathbb{D}(I)$ has a left adjoint, which is called the homotopy colimit $\textup{hocolim}_I: \mathbb{D}(I) \to \mathbb{D}(\mathbbm{1})$. 
\par 
This differs from the usual colimit, which is left adjoint to the constant diagram functor.
\begin{lemma}\label{agreementofcolim}
    Suppose $\mathcal{GP}^{\perp} = {}^{\perp}\mathcal{GI}$. Let $I$ be a directed set and $X \in \mathbb{D}(I)$. Then in the stable category there is an isomorphism $\textup{hocolim}_IX \simeq \varinjlim X$.
\end{lemma}
\begin{proof}
Let $\mathcal{M}$ be the model category we constructed in \Cref{ourbigmodel}. By definition, $\varinjlim_I $ is left adjoint to the constant diagram functor $\mathcal{M}^{\mathbbm{1}} \to \mathcal{M}^I$.
\par 
    Since $I$ is directed, we know from \cite[Theorem~5.1.3]{hovey2007model} that $\mathcal{M}^I$ may be endowed with the projective model structure, where weak equivalences and fibrations are determined levelwise. Furthermore, by \cite[Corollary~5.1.5]{hovey2007model} we know that $\varinjlim_I$ is a left Quillen functor, and hence preserves (trivial) cofibrations. 
    \par 
    We claim that $\varinjlim_I$ also preserves trivial fibrations. Recall from \Cref{mapsrmk} that in our model category, these are surjective maps with kernel in $\mathcal{GP}^{\perp}$. Since $\mathcal{GP}^{\perp} = {}^{\perp}\mathcal{GI}$ we infer from \cite[Lemma~4.1]{stovsarmodel} that $\mathcal{GP}^{\perp}$ is closed under direct limits. Furthermore, we know that taking direct limits is exact. From this, using that trivial fibrations in $\mathcal{M}^I$ are determined levelwise, we conclude that $\varinjlim_I$ preserves trivial fibrations as claimed. 
    \par
    Altogether, this implies that $\varinjlim_I$ preserves weak equivalences, and hence induces a functor $\textup{Ho}(\mathcal{M}^I) \to \textup{Ho}(\mathcal{M})$. The constant diagram functor $\mathcal{M} \to \mathcal{M}^I$ is right Quillen and hence preserves trivial fibrations; Ken Brown's Lemma \cite[Lemma~1.1.12]{hovey2007model} then implies that it preserves weak equivalences as well (since every object in $\mathcal{M}$ is fibrant). 
    \par
    Finally, we see that the induced functor $\varinjlim_I: \textup{Ho}(\mathcal{M}^I) \to \textup{Ho}(\mathcal{M})$ is still left adjoint to the $\textup{Ho}(\mathcal{M}) \to \textup{Ho}(\mathcal{M}^I)$; this follows just as \cite[Lemma~1.3.10]{hovey2007model} since we have shown that our functors are their own derived functors. However, this left adjoint is by definition the homotopy colimit; uniqueness of adjoints now proves the statement.
\end{proof}
Before we state the key theorem, we recall some terminology from \cite{SAORINSTOV}. 
\begin{definition}
    We say an object $X \in \mathbb{D}(\mathbbm{1})$ is intrinsically homotopically finitely presented if for every directed set $I$, the natural map 
    \[\varinjlim_I \mathbb{D}(I)(X,\mathscr{X}_i) \to \mathbb{D}(\mathbbm{1})(X, \textup{hocolim}_I\mathscr{X})\]
    is an isomorphism for each $\mathscr{X} \in \mathbb{D}(I)$.
\end{definition}
\begin{theorem}\label{idcomps}
    Suppose $\mathcal{GP}^{\perp} = {}^{\perp}\mathcal{GI}$. Then the compact objects in $\tacstab(kG)$ are the $\mathcal{GP}$-finitary modules. 
\end{theorem}
\begin{proof}
    It is shown in \cite[Corollary~6.10]{SAORINSTOV} that the compact objects in $\mathbb{D}(\mathbbm{1})$ are the intrinsically homotopically finitely presented objects. Using \Cref{agreementofcolim} we infer that these are simply the $\mathcal{GP}$-finitary objects. 
    \par Actually, we have technically only shown that the cohomology functors commute with directed colimits; however, as we mentioned after \Cref{extcrit}, it is shown in \cite[Corollary~1.5]{ARcats} that a functor which commutes with directed colimits must also commute with filtered colimits, which allows us to conclude that the statement of the theorem holds.
\end{proof}
\begin{corollary}\label{somecge}
    Suppose that $(\mathcal{GP},\mathcal{GP}^{\perp})$ is of finite type and that $\mathcal{GP}^{\perp} = {}^{\perp}\mathcal{GI}$. Then the stable category is compactly generated. This is true, in particular, for $\textup{LH}\mathfrak{G}_{AC}$ groups. 
\end{corollary}
\begin{proof}
    Take $M$ such that $\tachom_{kG}(C,M) = 0$ for every $\mathcal{GP}$-finitary module $C$. Let $X$ be an $\textup{FP}_{\infty}$ Gorenstein projective; every syzygy $\Omega^n(X)$ may be taken to be $\textup{FP}_{\infty}$, where we use \cite[Theorem~4.2(vi)]{ADTgor} for negative syzygies. We know that every $\textup{FP}_{\infty}$ Gorenstein projective is $\mathcal{GP}$-finitary by \Cref{fhsao}. Therefore, we know that $\textup{Ext}^1_{kG}(X,M) = 0$.
    \par 
    This is true for any $\textup{FP}_{\infty}$ Gorenstein projective. We now use that the cotorsion pair is of finite type to conclude that $M \in \mathcal{GP}^{\perp}$, and hence $M \cong 0$ in $\tacstab(kG)$. This implies compact generation by \Cref{lem:SSlocgen}. 
    \par 
    The fact that this is true for $\textup{LH}\mathfrak{G}_{AC}$ groups follows from \Cref{fintypec}.
\end{proof}
The following two corollaries of \Cref{idcomps} hold under the assumption that $\mathcal{GP}^{\perp} = {}^{\perp}\mathcal{GI}$.
\begin{corollary}\label{cor:fpgpfpinf}
Suppose $M$ is a finitely presented, Gorenstein projective $kG$-module. Then $M$ is $\textup{FP}_{\infty}$.
\end{corollary}
\begin{proof}
    It is shown in \cite[Lemma~1.2]{Eprecov} that $\underline{\textup{Hom}}_{kG}(M,-)$ commutes with filtered colimits for such an $M$. In particular, $M$ is compact in $\gp(kG)$. 
    \par 
    Hence, from \Cref{idcomps} we know that $M$ is $\mathcal{GP}$-finitary. Since $M$ is Gorenstein projective, we can see that $\textup{Ext}^i_{kG}(M,-)$ commutes with filtered colimits for all $i \geq 1$. 
    \par 
    Using the finite presentation of $M$ it is not hard to see that $\textup{Ext}^0_{kG}(M,-)$ also commutes with filtered colimits. We conclude from \Cref{extcrit} that $M$ is $\textup{FP}_{\infty}$. 
\end{proof}
\begin{corollary}\label{cute}
    Suppose $M$ is a finitely generated Gorenstein projective module. Then $M$ is $\textup{FP}_{\infty}$ and has a complete resolution which is finitely generated in each degree.
\end{corollary}
\begin{proof}
    Since $M$ is finitely generated we can see that $M$ is compact in $\gp(kG)$; this follows since any map from $M$ to a direct sum must factor through a finite sub-sum by virtue of the finite generation. Hence, by \Cref{idcomps} we know that $M$ is $\mathcal{GP}$-finitary. 
    \par 
    From \cite[Lemma~1.2]{Eprecov} we know that $M$ is a direct summand of $N \oplus P$ where $N$ is finitely presented and $P$ is projective. Since $M$ is finitely generated we may take $P$ to be finitely generated. It follows that $M$ is a summand of a finitely presented module and so is itself finitely presented. From \Cref{cor:fpgpfpinf} we see that $M$ is $\textup{FP}_{\infty}$.
    \par 
    It is then shown in \cite[Theorem~4.2(vi)]{ADTgor} how to construct a complete resolution which is finitely generated in each degree.
\end{proof}
This shows that Gorenstein projective modules have particularly nice properties, which are not shared by all modules as the following shows. 
\begin{example}\label{houghtonex}
    Let $n \geq 1$ and we write $H_n$ to denote Houghton's group. Briefly, this is the group of permutations of $\mathbb{N} \times \{0,\dots, n\}$ which are eventually translations, see \cite[Section~5]{BROWNHOUGHTON} for the precise definition. Brown \cite{BROWNHOUGHTON} has shown that this is of type $\textup{FP}_{n-1}$ but not $\textup{FP}_n$. In particular, the trivial module $\mathbb{Z}$ is finitely generated but not $\textup{FP}_{\infty}$. Note that $H_n$ is in $\textup{H}_1\mathfrak{F}$ and in fact has a finite dimensional model for the classifying space for proper actions, see for example the proof of \cite[Lemma~2.3]{kropfin} for the construction of such a model. 
\end{example}
\Cref{cute} has an interesting consequence for groups with $\textup{glGPD}(kG) < \infty$. Firstly, recall from \cite{BRAVOdcoh} that a ring $R$ is called $n$-coherent if every module $M$ which has a partial projective resolution $P_n \to P_{n-1} \to \dots \to P_0 \to M \to 0$ such that $P_i$ is a finitely generated projective for $0 \leq i \leq n$ is in fact $\textup{FP}_{\infty}$. This generalises the usual notions of noetherian rings, which are $0$-coherent in this language, and coherent rings, which are the $1$-coherent rings. 
\begin{corollary}\label{dcohco}
    Let $G$ be a group with $\textup{glGPD}(kG) < \infty$. Then $kG$ is an $n$-coherent ring, for some $n \geq 0$.
\end{corollary}
\begin{proof}
    Let $d = \textup{glGPD}(kG)$ and take $n = d + 1$. Take $M$ to be a module with a partial projective resolution $P_n \to P_{n-1} \to \dots \to P_0 \to M \to 0$ such that $P_i$ is a finitely generated projective for $0 \leq i \leq n$. This implies that $\Omega^d(M)$ is a finitely generated Gorenstein projective module. We conclude from \Cref{cute} that there exists a projective resolution of $\Omega^d(M)$ consisting of finitely generated projectives. Altogether, this shows that $M$ is itself $\textup{FP}_{\infty}$.
\end{proof}
Using this characterisation of compact modules we can give an alternative proof of \cite[Corollary~C]{demtalres}. For the following, recall Benson's model structure in \Cref{benstmodsthm}. The cofibrant modules in this model structure are those $M$ such that $M \otimes B$ is projective, where $B$ is the module of bounded functions $G \to k$. 
\begin{proposition}\label{ourcorC}
    Suppose that $G$ is an $\lhf$ group. Then Benson's cofibrant modules coincide with the Gorenstein projective modules. 
\end{proposition}
\begin{proof}
    It is known that every module which is cofibrant in Benson's model structure is Gorenstein projective, e.g. \cite[Proposition~4.1]{ADTgor}.
    \par On the other hand, suppose $M$ is a $\mathcal{GP}$-finitary Gorenstein projective. Since $B$ is free over each finite subgroup \cite[Lemma~3.4]{bensoninf1}, we know that $(M \otimes B){\downarrow}_F^G \in \mathcal{GP}_F^{\perp}$ for each finite subgroup $F \leq G$. Then we see from \Cref{corfsubfin} that $M \otimes B \in \mathcal{GP}_G^{\perp}$. 
    \par 
    This implies that $M \otimes B$ is isomorphic to zero in the homotopy category and since $M$ is compact, and the stable category is compactly generated, we know that $\underline{\textup{Hom}}_{kG}(M,M \otimes B) = 0$. Therefore, \cite[Proposition~3.5]{bensoninf1} shows that $M \otimes B$ must have finite projective dimension. Since this holds for any $\mathcal{GP}$-finitary Gorenstein projective, it is true for $\Omega^n(M)$ for all $n \in \mathbb{Z}$. We see that $M \otimes B$ must therefore be projective. 
    \par 
    Now, take an arbitrary Gorenstein projective $N$. We know from \Cref{induc} that the Gorenstein projective cotorsion pair is cogenerated by a collection of $\textup{FP}_{\infty}$ Gorenstein projectives (which are $\mathcal{GP}$-finitary by \Cref{fhsao}). It follows from \cite[Corollary~2.15]{Saorn2010OnEC} that every Gorenstein projective is a retract of a transfinite extension of modules in the cogenerating set. Since Benson's cofibrant modules are closed under retracts, we may assume that $N$ is a transfinite extension of $\mathcal{GP}$-finitary Gorenstein projectives. As $B$ is projective over $k$, we can tensor this transfinite extension to see that $N \otimes B$ has a transfinite extension by modules of the form $M \otimes B$ with $M$ $\mathcal{GP}$-finitary Gorenstein projective. Therefore, it is a transfinite extension of projective modules; the Eklof Lemma \cite[Lemma~1]{EkTrl}  implies that $N \otimes B$ must itself be projective. 
\end{proof}
We have shown that $\textup{FP}_{\infty}$ Gorenstein projectives are compact in the stable category, and so clearly any module with an $\textup{FP}_{\infty}$ Gorenstein projective approximation will also be compact. It is unclear, however, if a general $\textup{FP}_{\infty}$ module admits such an approximation. 
\par However, we can show that the $\textup{FP}_{\infty}$ modules, and more generally the completely finitary modules, are indeed compact in the stable category. Obviously, for groups of finite Gorenstein cohomological dimension, this follows immediately from \Cref{idcomps}, since we know from \Cref{thetwoext} that the completely finitary modules are the $\mathcal{GP}$-finitary modules in this case. 
\begin{proposition}
    Suppose $G$ is an $\textup{LH}\mathfrak{G}$ group and $M$ is a completely finitary $kG$-module. Then $M$ has finite Gorenstein projective dimension. In particular, $M$ is compact in the stable category. 
\end{proposition}
\begin{proof}
    Consider the Gorenstein projective special precover $0 \to K \to X \to M \to 0$, so that $K \in \mathcal{GP}^{\perp}$ and $X \in \mathcal{GP}$. Applying \Cref{corfsub} we know that $K{\downarrow}_H^G \in \mathcal{GP}^{\perp}$ for each $\mathfrak{G}$-subgroup $H \leq G$. Furthermore, this implies, by \Cref{fpdext}, that $K{\downarrow}_H^G$ has finite projective dimension. Because we have assumed that $M$ is completely finitary, we may apply \cite[Lemma~3.3]{krop} to the functors $\widehat{\textup{Ext}}^*_{kG}(M,-)$ to conclude that $\widehat{\textup{Ext}}^0_{kG}(M,K) = 0$. 
    \par 
    In fact, each positive syzygy of $M$ must also be completely finitary, and so a dimension shifting argument shows that $\widehat{\textup{Ext}}^i_{kG}(M,K) = 0$ for all $i \geq 0$. This implies that $\widehat{\textup{Ext}}^0_{kG}(X,K) \cong \widehat{\textup{Ext}}^0_{kG}(K,K)$. 
    \par 
    However, $X$ is Gorenstein projective and so there exists a short exact sequence $0 \to X \to P \to X' \to 0$ with $P$ projective and $X'$ Gorenstein projective. Therefore, we know that \[\widehat{\textup{Ext}}^0_{kG}(X,K) \cong \widehat{\textup{Ext}}^1_{kG}(X',K) \cong \textup{Ext}^1_{kG}(X',K) = 0\] Here the second isomorphism is using \eqref{gpla1} and the equality is by assumption on $K$. Altogether this shows that $\widehat{\textup{Ext}}^0_{kG}(K,K) = 0$. It is well known that this implies that $K$ has finite projective dimension, see \cite[Theorem~4.2]{krop} for a proof. 
    \par 
    We can splice together a projective resolution of $K$ with the Gorenstein projective special precover $X \to M \to 0$ to show that $M$ has finite Gorenstein projective dimension as claimed. 
    \par 
    The final claim now follows from \Cref{thetwoext} and \Cref{idcomps}. 
\end{proof}
This has the following corollary, which is a sort of weak form of Kropholler and Mislin's theorem \cite{kropmis} that every $\lhf$ group of type $\textup{FP}_{\infty}$ must have a finite dimensional model for the classifying space for proper actions. 
\begin{corollary}
    Let $G$ be an $\textup{LH}\mathfrak{G}$ group of type $\textup{FP}_{\infty}$. Then $G$ has finite Gorenstein cohomological dimension. 
\end{corollary}
The following examples show that not every compact object has to be completely finitary in general. 
\begin{example}\label{freabex}
  Let $G$ be free abelian group of countable rank. It is not hard to see using \Cref{corfsubfin} that the stable category is zero, since $G$ is torsion free. We see that trivially $\mathbb{Z}$ is a compact object. If $\mathbb{Z}$ were completely finitary, then it would have finite Gorenstein projective dimension. This is a contradiction, since $\mathbb{Z}$ doesn't have a complete resolution \cite[Example~5.3]{CKcompres}.
  \par 
  If we let $G' = G \ast C_2$ then the stable category is now non-zero. The short exact sequence $\mathbb{Z}{\uparrow}_1^{G'} \to \mathbb{Z}{\uparrow}_G^{G'} \oplus \mathbb{Z}{\uparrow}_{C_2}^{G'} \to \mathbb{Z}$ allows us to conclude that there is a stable isomorphism $\mathbb{Z} \cong \mathbb{Z}{\uparrow}_{C_2}^{G'}$, which is compact. However, as before we see that $\mathbb{Z}$ is not completely finitary.
\end{example}
\subsection{Compactness of the tensor unit}
Another difference between finite and infinite groups is that the tensor unit $k$ doesn't have to be a compact object in the stable category. The following gives an abundance of examples of groups without a compact tensor unit.
\begin{lemma}\label{boundlem}
    Suppose $G$ is an $\textup{H}_1\mathfrak{F}$ group such that $\mathbb{Z}$ is compact in the stable category. Then there is a bound on the orders of the finite subgroups of $G$.
\end{lemma}
\begin{proof}
    We use \Cref{thetwoext} and \Cref{idcomps} to see that under our assumptions the complete cohomology $\widehat{\textup{Ext}}^0_{\mathbb{Z}G}(\mathbb{Z},-)$ commutes with filtered colimits. It is now an argument of Kropholler, see for example \cite[Theorem~2.7(vii)]{kropfin}, that this implies that $G$ has a bound on the order of the finite subgroups.
\end{proof}
\begin{example}
    Consider $\bigast\limits_{n \in \mathbb{N}} C_n$, which acts on a tree with finite stabilisers, and hence is in $\textup{H}_1\mathfrak{F}$, but has no bound on the orders of the finite subgroups. Therefore, $\mathbb{Z}$ is not compact in the stable category.
\end{example}
For groups $G$ such that $\textup{glGPD}_{AC}(kG) < \infty$, we now give a characterisation of when $k$ is compact in terms of complete cohomology; as discussed above this is the cohomology theory defined by the complete cotorsion pair $(\mathcal{GP},\mathcal{GP}^{\perp})$ in this case. 
\par 
In order to do this, we have the following description of compact Gorenstein projectives. This should be compared to the well known fact that for a finite group and a field, the compact objects in the stable module category are exactly the finitely generated modules.
\begin{proposition}\label{thicksubcateescom}
Suppose $k$ is a commutative noetherian ring of finite global dimension and let $\mathfrak{G}_{AC}$ be the class of all groups with $\textup{glGPD}_{AC}(kG) < \infty$. Then for any $\textup{LH}\mathfrak{G}_{AC}$ group $G$, the compact objects in $\gp(kG)$ are those which are summands in $\gp(kG)$ of a finitely generated Gorenstein projective. 
\end{proposition}
\begin{proof}
We know from \Cref{somecge} that in this situation the stable category is compactly generated by $\textup{FP}_{\infty}$ Gorenstein projective modules. 
    We claim that the compact objects are summands in $\gp(kG)$ of an $\textup{FP}_{\infty}$ Gorenstein projective. The statement then follows from \Cref{cute}. 
    \par 
    It is shown in \cite[Proposition~3.4.15]{krause_2021} that every compact object is a summand of an extension of objects in the generating set of compact objects. Therefore, it suffices to show that an extension of modules which are isomorphic in $\gp(kG)$ to an $\textup{FP}_{\infty}$ module is itself isomorphic in $\gp(kG)$ to an $\textup{FP}_{\infty}$ module. 
    \par 
    Suppose we have an exact triangle in $\gp(kG)$ of the form $M \to X \to N \to \Omega^{-1}(M)$ with $M$ and $N$ both isomorphic in $\gp(kG)$ to an $\textup{FP}_{\infty}$ Gorenstein projective. By replacing the triangle with an isomorphic triangle, we may assume that $M$ and $N$ are actually $\textup{FP}_{\infty}$. 
    \par
    We let $\mathcal{D}$ be the full subcategory of $\textup{GP}(kG)$ consisting of all $\textup{FP}_{\infty}$ Gorenstein projectives. In particular, this contains all finitely generated projectives and we let $\underline{\mathcal{D}}$ be the quotient by all finitely generated projectives. There is a natural functor $F: \underline{\mathcal{D}} \to \gp(kG)$. Now, we note that if a morphism between two $\textup{FP}_{\infty}$ modules factors through a projective, it must factor through a finitely generated projective. This implies that $F$ is fully faithful. 
    \par 
    Since $M$ is $\textup{FP}_{\infty}$, it follows from \Cref{cute} that we may take $\Omega^{-1}(M)$ to be $\textup{FP}_{\infty}$. This means that the morphism $N \to \Omega^{-1}(M)$ lifts to a morphism in $\underline{\mathcal{D}}$. We complete this, inside $\underline{\mathcal{D}}$, to a triangle $X' \to N \to \Omega^{-1}(M) \to \Omega^{-1}(X')$ and note that $X'$ is isomorphic in $\gp(kG)$ to $X$. Therefore, we are done once we show that $X'$ is $\textup{FP}_{\infty}$.
    \par 
    From the construction of triangles in $\underline{\mathcal{D}}$ (see \cite{happel_1988}), we have a short exact sequence $0 \to N \to \Omega^{-1}(M) \oplus P(N) \to \Omega^{-1}(X') \to 0$, where $P(N)$ is a projective into which $N$ embeds. We clearly may take $P(N)$ to be finitely generated. Using Brown's criterion for being $\textup{FP}_{\infty}$ in \Cref{extcrit}, it is not hard to see that $\Omega^{-1}(X')$ is also $\textup{FP}_{\infty}$. This implies that $X'$ is $\textup{FP}_{\infty}$ as desired.
\end{proof}
\begin{lemma}\label{oooo}
    Suppose $k$ is a commutative ring of finite global dimension and $G$ is a group with $\textup{glGPD}_{AC}(kG) < \infty$. Then every $\mathcal{GP}$-finitary $kG$-module is a summand of a module which has a free resolution which is finitely generated in all high enough degrees.
\end{lemma}
\begin{proof}
    \sloppy Suppose $M$ is $\mathcal{GP}$-finitary and let $F_* \to M$ be any free resolution. Since $\textup{glGPD}_{AC}(G) < \infty$, for some $d \geq 0$ we know that $\Omega^d(M)$ is Gorenstein projective and it must also be $\mathcal{GP}$-finitary. It follows from \Cref{thicksubcateescom} that $\Omega^d(M)$ is isomorphic in $\gp(kG)$ to a summand of a Gorenstein projective module $X$ which is $\textup{FP}_{\infty}$. Therefore, for some Gorenstein projective $N$ and projectives $P$ and $Q$ there is a genuine $kG$-module isomorphism $\Omega^d(M) \oplus N \oplus Q \cong X \oplus P$ cf. \cite[Lemma~3.1]{MSstabcat}. We let $F'_*$ be a finitely generated free resolution of $X$ and let $F''_*$ be a totally acyclic complex such that $N$ is the kernel in degree $0$. 
    \par 
    We then have the following exact sequence.
    \par 
    \[\begin{tikzcd}
           \dots \arrow{r}& F'_1 \arrow{r} & F'_0 \oplus P \arrow{r}& F_d \oplus F''_0 \oplus Q  \arrow{r} & F_{d-1} \oplus F''_{-1} \arrow{r} & \dots \\ \dots \arrow{r}& F_0 \oplus F''_{-d}\arrow{r} & M \oplus \Omega^{-d}(N)\arrow{r} & 0 
    \end{tikzcd}\]
    Using Eilenberg's swindle we find a free module $F$ such that $F'_0 \oplus P \oplus F \cong F$ (and similarly for $F_d \oplus F''_0 \oplus Q$) and so we may replace these by free modules. Altogether, this gives us a free resolution of $M \oplus \Omega^{-d}(N)$ which is finitely generated in all high enough degrees.
\end{proof}
The following is then a variant of \cite[Theorem~A]{finconCEKT}, where we remove the hypothesis of the group being in a certain hierarchy of groups, although we require that $\textup{glGPD}_{AC}(kG) < \infty$ rather than $\textup{Gcd}_k(G) < \infty$. Note that if $\gp(kG)$ is compactly generated for groups with $\textup{Gcd}_k(G) < \infty$, then the following would hold for such groups. It gives a topological interpretation of the trivial representation being compact in the stable category. 
\begin{theorem}\label{compunth}
    Suppose $G$ is a group with $\textup{glGPD}_{AC}(\mathbb{Z}G) < \infty$. Then the following are equivalent. 
    \begin{enumerate}[label=(\roman*)]
    \item The trivial representation $\mathbb{Z}$ is compact in $\tacstab(\mathbb{Z}G)$
        \item The complete cohomology $\hat{H}^*(G,-)$ commutes with filtered colimits of modules
        \item The group $G$ has an Eilenberg-Maclane space $K(G,1)$ which is dominated by a CW-complex with finitely many $n$-cells for all large enough $n > 0$
        \item $\mathbb{Z}$ is a summand of a module which has a free resolution which is finitely generated in all large enough degrees
    \end{enumerate}
\end{theorem}
\begin{proof}
    The equivalence of $(i)$ and $(ii)$ follows from \Cref{idcomps}, using that $\textup{Ext}^n_{\mathcal{GP}}(\mathbb{Z},-)$ is isomorphic to the complete cohomology by \Cref{thetwoext}. 
    \par 
    Suppose that $(ii)$ holds and so $\mathbb{Z}$ is completely finitary. Using \Cref{oooo} the argument that $(iii)$ holds exactly as in \cite[Theorem~3.9]{finconCEKT}. 
   \par
    The implication $(iii) \implies (ii)$ holds for any group, as shown in \cite[Proposition~2.24]{Hamfin}.
    \par 
    Finally, the fact that $(i)$ and $(iv)$ are equivalent follows from \Cref{oooo}.
\end{proof}
From \cite[Corollary~2.6]{MSstabcat} we know that any group with a finite dimensional model for the classifying space for proper actions is of type $\Phi$. In the following, by cocompact we mean finite dimensional and of finite type (i.e. finitely many G-orbits in each dimension).
\begin{proposition}
    Suppose $G$ has a cocompact model for the classifying space for proper actions. Then $\mathbb{Z}$ is compact in $\tacstab(\mathbb{Z}G)$.
\end{proposition}
\begin{proof}
    The augmented cellular chain complex of the assumed model is of the form $0 \to C_n \to \dots \to C_0 \to \mathbb{Z} \to 0$ where each $C_i$ is a finite direct sum of modules $\mathbb{Z}{\uparrow}_F^G$ for some finite subgroup $F \leq G$. Hence, this is a resolution of $\mathbb{Z}$ given by finitely generated Gorenstein projectives. 
    \par 
    We know from \Cref{cute} that each finitely generated Gorenstein projective is $\textup{FP}_{\infty}$. Also, the cokernel of a monomorphism between $\textup{FP}_{\infty}$ modules is also $\textup{FP}_{\infty}$ - this is easy to see from \Cref{extcrit}. It follows that $\mathbb{Z}$ is itself $\textup{FP}_{\infty}$ and hence compact.
\end{proof}
\begin{proposition}\label{fpgobo}
    Suppose $G$ is a group of type $\Phi_{\mathbb{Z}}$ or an $\lhf$ group. Suppose that $G$ has finitely many conjugacy classes of finite subgroups and the normaliser of each nontrivial finite subgroup is $\textup{FP}_{\infty}$. Then $\mathbb{Z}$ is compact in $\tacstab(\mathbb{Z}(G))$.
\end{proposition}
\begin{proof}
    We follow the idea of \cite[Theorem~3.8]{HAMILTON2011}. Consider the simplicial complex $\Delta$ constructed from chains of nontrivial finite subgroups of $G$. By assumption there must be a bound on the orders of the finite subgroups and so this is of finite dimension and we call its dimension $r$. From \cite[Proposition~3.4]{HAMILTON2011} we can embed $\Delta$ into an $r$-dimensional $G$-CW complex $Y$ which is $(r-1)$-connected and is such that $G$ acts freely outside of $\Delta$. Therefore, we have the following exact sequence. 
    \[0 \to \tilde{H}_r(Y) \to C_r(Y) \to \dots \to C_0(Y) \to \mathbb{Z} \to 0\]
    As in the proof of \cite[Theorem~3.8]{HAMILTON2011}, we see that $\tilde{H}_r(Y){\downarrow}_F^G$ has finite projective dimension for all finite subgroups and hence is isomorphic to $0$ in the stable category by \Cref{corfsubfin}. 
    \par 
    We wish to show that each $C_i(Y)$ is compact in the stable category for each $0 \leq i \leq r$, which would imply that $\mathbb{Z}$ is compact as desired. Since $G$ acts freely on $Y$ outside of $\Delta$, it follows that $C_i(Y)$ and $C_i(\Delta)$ differ only by free modules. Therefore, it suffices to show that each $C_i(\Delta)$ is compact. 
    \par 
    We see that each $C_i(\Delta)$ is isomorphic to a direct sum of permutation modules $\mathbb{Z}{\uparrow}_N^G$, where $N$ is the intersection of the normalisers of the given chain of finite subgroups and the direct sum is taken over a set of representatives for conjugacy classes of finite subgroups. By assumption, this direct sum is finite. We are therefore reduced to showing that if $F_0 \leq \dots \leq F_i$ is a chain of finite subgroups and $N = \bigcap\limits_{j = 0}^iN_{G}(F_j)$ then $\mathbb{Z}{\uparrow}_{N}^G$ is compact. 
    \par 
    It is easy to see that $C_G(F_i) \leq N \leq N_G(F_i)$ and hence $N$ must have finite index in $N_G(F_i)$. By assumption, we know that $N_G(F_i)$ is a group of type $\textup{FP}_{\infty}$ and therefore $N$ must also be of type $\textup{FP}_{\infty}$. It then follows that $\mathbb{Z}{\uparrow}_N^G$ is $\textup{FP}_{\infty}$ and hence is compact.
\end{proof}
\begin{proposition}\label{prop:fdmod}
    Let $G$ be an $\lhf$ group of type $\Phi_{\mathbb{Z}}$. Assume that $\mathbb{Z}$ is compact in $\tacstab(\mathbb{Z}G)$. Then $G$ has a finite dimensional model for the classifying space for proper actions.
\end{proposition}
\begin{proof}
    We have identified the compact objects with $\mathcal{GP}$-finitary modules. Furthermore, since groups of type $\Phi$ have finite Gorenstein cohomological dimension, the $\mathcal{GP}$-finitary modules are exactly the completely finitary modules. In this language, the statement is now just \cite[Theorem~2.7]{kropfin}. 
\end{proof}
The converse of the above does not hold, as the follow examples show. 
\begin{example}
    Consider $G = C_2 \times \mathbb{Q}^+$. This has finite virtual cohomological dimension and it is clearly abelian-by-finite and therefore it is locally polycyclic-by-finite. It follows from \cite[Lemma~3]{HAMILTON2011} that it has a finite dimensional model for the classifying space for proper actions. However, as it is abelian, the normaliser of the copy of $C_2$ is the whole group, which is infinitely generated. Therefore, by \cite[Theorem~2]{HAMILTON2011} it cannot have almost everywhere finitary cohomology over $\mathbb{Z}$ and hence the trivial module $\mathbb{Z}$ is not compact.
\end{example}
\begin{example}
    Consider Houghton's group $H_n$ as in \Cref{houghtonex}. As mentioned, this has a finite dimensional model for the classifying space for proper actions. Note that there can be no bound on the orders of the finite subgroups of $H_n$; it contains, for example, the symmetric group $S_m$ for each $m \geq 1$. From \Cref{boundlem} we now see that $\mathbb{Z}$ cannot be a compact object in the stable category. Note that for $n \geq 3$ Brown \cite{BROWNHOUGHTON} has shown that $H_n$ is finitely presented. In particular, $\mathbb{Z}$ is a finitely presented module which is not compact.
\end{example}
We also cannot remove the requirement that $G$ be of type $\Phi$ in \Cref{prop:fdmod}.
\begin{example}
Let $G$ be the free abelian group of countable rank as in \Cref{freabex}. It is not hard to check that this is not of type $\Phi$, e.g. \cite[Example~2.7]{MSstabcat}. Hence, there can be no finite dimensional model for the classifying space for proper actions (since this would imply it is of type $\Phi$ \cite[Corollary~2.6]{MSstabcat}). However, as we mentioned in \Cref{freabex} the trivial module $\mathbb{Z}$ is a compact object in the stable category.
\par 
For an example when the stable category is non-zero, we can again consider the group $G' = C_2 \ast G$. This has no finite dimensional model for the classifying space for proper actions but $\mathbb{Z}$ is a compact object as we showed in \Cref{freabex}.
\end{example}
    If type $\Phi$ is a Weyl group closed property (by which we mean that if $G$ is type $\Phi$ so is $N_G(F)/F$ for every finite subgroup $F \leq G$) then, by carefully inspecting the proof of \cite[Theorem~2.7]{kropfin}, we can see that the assumption of $G$ being in $\lhf$ may be dropped in \Cref{prop:fdmod}. 
  \par

We can also compare our finiteness conditions to the homotopy theoretic stable finiteness conditions in \cite{homotopything}. We refer to that paper for the definition of the category $\textup{Ho(Sp}_G)$. 
\begin{proposition}\label{cdac}
    Let $G$ be a group of type $\Phi_{\mathbb{Z}}$. If the sphere spectrum is a compact object in $\textup{Ho(Sp}_G)$ then the trivial module is compact in $\tacstab(\mathbb{Z}G)$. However, the converse does not hold.
\end{proposition}
\begin{proof}
    If the sphere spectrum is compact in $\textup{Ho(Sp}_G)$ then by \cite[Theorem~5.1]{homotopything} it follows that there are only finitely many conjugacy classes of finite subgroups of $G$ and the Weyl group of each finite subgroup, that is $N_G(F)/F$, is $\textup{FP}_{\infty}$. Since $F$ is of type $\textup{FP}_{\infty}$ it follows from \cite[Proposition~2.7]{bieri1981homological} that $N_G(F)$ is also $\textup{FP}_{\infty}$. From \Cref{fpgobo} we know that $\mathbb{Z}$ is compact in $\tacstab(\mathbb{Z}G)$. 
    \par 
    In fact, this implies that the Weyl group of the trivial subgroup, which is just $G$, is also $\textup{FP}_{\infty}$. 
    \par 
    To show the converse does not hold, consider the semidirect product $G = \mathbb{Q}^+ \rtimes C_2$ where the nontrivial element of $C_2$ acts on $\mathbb{Q}^+$ by multiplication by $-1$. In particular, this is abelian by finite and of finite virtual cohomological dimension. It is not hard to see that the normaliser of every finite subgroup is finitely generated, and so by \cite[Theorem~20]{HAMILTON2011} it follows that $\mathbb{Z}$ is completely finitary (which corresponds to being $\mathcal{GP}$-finitary since groups of type $\Phi$ have finite Gorenstein cohomological dimension) and therefore it is compact in $\tacstab(\mathbb{Z}G)$. However, $G$ is not finitely generated and so is not $\textup{FP}_{\infty}$. We have shown above that if the sphere spectrum is compact in $\textup{Ho(Sp}_G)$ then the group is $\textup{FP}_{\infty}$ and so we conclude that the sphere spectrum is not compact in $\textup{Ho(Sp}_G)$.
\end{proof}

\bibliographystyle{amsplain}

\bibliography{infgrpref} 

\end{document}